\documentclass[11pt]{article}

\usepackage{subfigure}

\usepackage{amsmath, amsthm, amssymb,mathrsfs}
\input epsf.tex
\usepackage{epsf}
\usepackage{epsfig}

\usepackage[vcentering,dvips]{geometry}
\newtheorem{theorem}{Theorem}[section]
\newtheorem{lemma}[theorem]{Lemma}
\newtheorem{prop}[theorem]{Proposition}
\newtheorem{corollary}[theorem]{Corollary}

\newtheorem{remark}{Remark}

\newcommand{\rank}{\operatorname{rank}}

\newcommand{\s}{\sigma}
\newcommand{\diag}{\operatorname{diag}}
\newcommand{\diam}{\operatorname{diam}}
\newcommand{\supp}{\operatorname{supp}}
\newcommand{\argmin}{\arg \min}

\newcommand{\R}{\mathbb{R}}

\usepackage[usenames]{color}
\usepackage{fullpage}

\begin{document}

\title{Learning Functions of Few Arbitrary Linear Parameters in High Dimensions}
\author{
Massimo Fornasier\footnote{Technical University of Munich, Faculty of Mathematics, Boltzmannstra\ss e 3, D-85748 Garching, Germany,
email: {\tt 	massimo.fornasier@ma.tum.de}.}, \  
Karin Schnass\footnote{Johann Radon Institute for Computational and 
Applied Mathematics, Austrian Academy of Sciences, Altenbergerstra\ss e 69, A-4040 Linz, Austria,
email: {\tt  karin.schnass@oeaw.ac.at}.} \ and
Jan Vybiral\footnote{Johann Radon Institute for Computational and 
Applied Mathematics, Austrian Academy of Sciences, Altenbergerstra\ss e 69, A-4040 Linz, Austria,
email: {\tt  jan.vybiral@oeaw.ac.at}.}
}
\date{\it Dedicated to Ronald A. DeVore for his $70^{th}$ birthday}

\maketitle

\begin{abstract}
Let us assume that $f$ is a continuous function defined on the unit ball of $\mathbb R^d$, 
of the form $f(x) = g (A x)$, where $A$ is a $k \times d$ matrix and $g$ is a function of $k$ 
variables for $k \ll d$. We are given a budget $m \in \mathbb N$ of possible point evaluations 
$f(x_i)$, $i=1,\dots,m$, of $f$, which we are allowed to query in order to construct a uniform 
approximating function. Under certain smoothness and variation assumptions on the function $g$, 
and an {\it arbitrary} choice of the matrix $A$, we present in this paper 

1. a sampling choice of the points $\{x_i\}$ drawn at random for each function approximation; 

2. algorithms (Algorithm 1 and Algorithm 2) for computing the approximating function, whose complexity is at most polynomial in the dimension $d$ and in the number $m$ of points.


Due to the arbitrariness of $A$, the choice of the sampling points will be according to suitable random distributions and our results hold with overwhelming probability.
Our approach uses tools taken from the {\it compressed sensing} framework, recent Chernoff bounds for sums of positive-semidefinite matrices, and classical stability bounds for invariant subspaces of singular value decompositions.
\end{abstract}

\noindent
{\bf AMS subject classification (MSC 2010):} 65D15, 03D32, 68Q30, 60B20, 60G50
\\



\section{Introduction}
\subsection{Learning high dimensional functions from few samples}
In large scale data analysis and learning, several real-life problems can be formulated as capturing or approximating a function 
defined on $\Omega \subset \mathbb R^d$ with dimension $d$ very large, from relatively few given samples or queries. The usual
 assumption on the class of functions to be recovered is smoothness. The more regular a function is, the more accurately and 
 the more efficiently it can be numerically approximated. 
However, in the field of {\it information based complexity} it has been clarified that such a problem is in general {\it intractable}, 
i.e., it does not have polynomial complexity. To clarify this poor approximation phenomenon, assume
$$
\mathcal F_d:=\{ f:[0,1]^d \to \mathbb R, \| D^\alpha f\|_\infty \leq 1, \alpha \in \mathbb N_0^d\},
$$
to be the class of smooth functions we would like to approximate. We define the sampling operator $S_n = \phi \circ N$, 
where $N:\mathcal F_d \to \mathbb R^n$ is a suitable measurement operator and $\phi:\mathbb R^n \to L_\infty([0,1]^d)$ a 
recovery map. For example $N$ can take $n$ samples $f(x_i)$, $i=1,\dots,n$ of $f$ and $\phi$ can be a suitable interpolation operator.
The approximation error provided by such a sampling operator is given by
$$
e(S_n) := \sup_{f \in \mathcal F_d} \| f - S_n(f) \|_\infty.
$$
With this notion we further define the approximation numbers
$$
e(n,d) := \inf_{S_n} e(S_n),
$$
indicating the performance of the best sampling method, and 
\begin{equation}\label{eq:infcompl}
n(\varepsilon, d) := \inf\{n: e(n,d)\leq \varepsilon \},
\end{equation}
which is the minimal number of samples we need for the best sampling method to achieve a uniform accuracy $\varepsilon \in (0,1)$.

\subsection{Intractability results}
Recent results by Novak and Wo\'zniakowski \cite{nowo09} state that for a uniform approximation over $\mathcal F_d$ we 
have $e(n,d)=1$ for all $n \leq 2^{\lfloor d/2 \rfloor}-1$ or $n(\varepsilon, d) \geq 2^{\lfloor d/2 \rfloor}$ for all 
$\varepsilon \in (0,1)$. Hence, the number of samples to approximate even a $C^\infty$-function grows exponentially 
with the dimension $d$. This result seems to obliterate any hope for an efficient solution of the learning problem in 
high dimension, and this phenomenon is sometimes referred to as the {\it curse of dimensionality}. \\
Nevertheless, very often the high dimensional functions which we can expect as solutions to real-life problems exhibit 
more structure and eventually are much better behaved with respect to the approximation problem. There are several 
models currently appearing in the literature for which the approximation problem is {\it tractable}, i.e., the 
approximation error does not grow exponentially with respect to the dimension $d$.

According to the behavior of the \emph{information complexity} $n(\varepsilon,d)$, cf. \eqref{eq:infcompl},
for small $\varepsilon>0$ and large $d\in {\mathbb N}$, one speaks about
\begin{itemize}
\item \emph{polynomial tractability}: if $n(\varepsilon,d)$ depends polynomially on $\varepsilon^{-1}$ and $d$,
\item \emph{strong polynomial tractability}: if $n(\varepsilon,d)$ depends polynomially only on $\varepsilon^{-1}$,
\item \emph{weak tractability}: if $\displaystyle\lim_{\varepsilon^{-1}+d\to \infty} \frac{\log n(\varepsilon,d)}{\varepsilon^{-1}+d}=0$.
\end{itemize}
We point to \cite[Chapters 1 and 2]{nowo08} for further notions of tractability and many references.

In the next two subsections we will recount a few relevant approaches leading in some cases to (some sort of) tractability.
\subsection{Functions of few variables}
A function $ f:[0,1]^d \to \mathbb R$ of $d$ variables  ($d$ large) may be a sum of functions, which only depend on $k$ variables ($k$ small):
\begin{equation}
\label{model0}
f(x_1, \dots, x_d) = \sum_{\ell=1}^m g_\ell(x_{i_1},\dots, x_{i_k}).
\end{equation}
In optimization such functions are called {\it partially separable}. This model arises for instance in physics, when we consider problems involving interaction potentials, such as the Coulomb potential in electronic structure computations, or in social and economical models describing multiagent dynamics. Once $k$ is fixed and $d \to \infty$, the learning problem of such functions is tractable, even if the $g_\ell$ are not very smooth. We specifically refer to the recent work of DeVore, Petrova, and Wojtaszczyk \cite{depewo09} which describes an adaptive method for the recovery of high dimensional functions in this class, for $m=1$. \\
This model can be extended to functions which are only approximatively depending on few variables, by considering the unit ball $\mathcal H_{d,\gamma}$ of the weighted Sobolev space of functions $ f:[0,1]^d \to \mathbb R$ with
\begin{equation}
\label{model2}
\| f\|_{d,\gamma}^2 := \sum_{ u \subset [d]} \gamma_{d,u}^{-1} \int_{[0,1]^d} \left ( \frac{\partial^{|u|}}{\partial x_u} f(x) \right)^2 dx \leq 1, 
\end{equation}
where $[d] := \{1,\dots,d\}$, and $\gamma:=\{\gamma_{d,u}\}$ are non-negative weights; the definition $\frac{0}{0}:=0$ and the choice of $\gamma_{d,u}=0$ leads us again to the model \eqref{model0}. A study of the tractability of this class, for various weights, can be found in \cite{nowo08}.

\subsection{Functions of one linear parameter in high dimensions}
One of the weaknesses of the model classes introduced above is that they are very coordinate biased. It would be desirable to have results for a class of basis changes which would make the model basis-independent. A general model assumes that,
\begin{equation} \label{model1}
f(x) = g (A x),
\end{equation}
for $A$ an arbitrary $k\times d$ matrix.
While solution to this unconstrained problems have so far been elusive, the special case of
\begin{equation}\label{model1'}
f(x) = g (a \cdot x),
\end{equation}
where $a$ is a stochastic vector, i.e., $a = (a_1,\dots,a_d)$, $a_j \geq 0$, $\sum_{j=1}^d a_j =1$, and $g:[0,1] \to \mathbb R$ is a $\mathcal C^s$ function for $s>1$ has been fully addressed with an optimal recovery method in \cite{codadekepi10}.\\

The aim of this work is to find an appropriate formulation of the general model \eqref{model1}, which generalizes both the model of $k$ active coordinates as well as the model of one stochastic vector, and to analyze the tractability of the corresponding approximation problem. The rest of the paper is organized as follows. After introducing some basic notations, the next section is dedicated to the motivation and discussion of the generalized model. As an introduction to our formulation and solution approach, we then proceed to analyze the simple case of one active direction in Section~\ref{sec:kis1}, under milder assumptions on the vector $a=(a_1, \dots,a_d)$, before finally addressing the fully generalized problem in  Section~\ref{sec:klarger1}. The last section is dedicated to the discussion of further extensions of our approach, to be addressed in successive papers.


\subsection{Notations}
In the following we will deal exclusively with real matrices and we denote the space of $n \times m$ real matrices by $M_{n \times m}$. 
The entries of a matrix $X$ are denoted by lower case letters and the corresponding indices, i.e., $X_{ij}=x_{ij}$. The transposed matrix 
$X^T \in M_{m \times n}$ of a matrix $X \in M_{n \times m}$ is the matrix with entries $x_{ij}^T = x_{ji}$.
For $X \in   M_{n \times m}$ we can write its (reduced) {\it singular value decomposition} \cite{gova96} as
$$
 X= U \Sigma V^T
$$
with $U \in M_{n\times p}$, $V \in M_{m \times p}$, $p\leq \min(n,m)$, matrices with orthonormal columns and $\Sigma = \diag(\s_1,\dots,\s_p) \in M_{p \times p}$ a
diagonal matrix where $\s_1 \geq \s_2\geq \dots \geq \s_p \geq 0$ are the {\it singular values}. For specific 
matrices $X$ we write the singular value decomposition
$$
 X= U(X) \Sigma(X) V(X)^T = U_X \Sigma_X V_X^T.
$$
For symmetric, positive semidefinite matrices, i.e., $X = X^T$ and $v^T Xv\geq0$ for all vectors $v$, we can take $V=U$ and the singular value decomposition is equivalent to the eigenvalue decomposition. Note also that $\s_i(X) = \sqrt{\lambda_i(X^T X)}$, where $\lambda_i(X^T X)$ is the $i^{th}$ largest eigenvalue of the matrix $X^T X$ (actually, this holds for $n \geq m$, whereas we may want to consider $X X^T$ instead of $X^T X$ if $m >n$). The rank of $ X \in M_{n \times m}$ 
denoted by $\rank(X)$ is the number of nonzero singular values. 
We define the Frobenius norm of a matrix $X$ as
$$
\| X\|_F := \left ( \sum_{ij} |x_{ij}|^2 \right )^{1/2}.
$$
It is also convenient to introduce the $\ell_p^n$ vector norms 
$$
\| x\|_{\ell_p^n} := \left ( \sum_{i=1}^n |x_i|^p \right)^{1/p}, \quad 0 < p < \infty.
$$
We denote by $I_n \in M_{n \times n}$ the identity matrix. The symbol $B_{\mathbb R^n}$ stands for the unit ball
and $B_{\mathbb R^n}(r)$ for the ball of radius $r>0$ in $\mathbb R^n$. 
The unit sphere in $\mathbb R^n$ is denoted by $\mathbb S^{n-1}$. Finally, $\mathcal L^n$ indicates the Lebesgue measure in $\mathbb R^n$.

\section{The General Model $f(x) = g( A x)$ and Its Simplifications}

The first approach one may be tempted to consider to a generalization of \eqref{model1'}  is to ask that $f:[0,1]^d \to \mathbb R$ is of the form $f(x)=g(Ax)$,
where $A$ is a $k\times d$ stochastic matrix with orthonormal rows, i.e., $a_{ij} \geq 0$, $\sum_{j=1}^d a_{ij} =1$ for all $i=1,\dots,k$, $AA^T=I_k$, and $g:A([0,1]^d) \to \mathbb R$ is a $\mathcal C^s$ function for $s>1$. 
There are however two main problems with this formulation. The conditions of stochasticity and orthonormality of the rows of $A$ together are very restrictive - the only matrices satisfying both of them are those having only one non-negative entry per column - and the domain of $g$ cannot be chosen generically as $[0,1]^k$ but depends on $A$, i.e., it is the k-dimensional polytope $A([0,1]^d)$.
Thus we will at first return to the unconstrained model in~\eqref{model1} and give up the conditions of stochasticity and orthonormality. This introduces rotational invariance for the rows of A and the quadrant defined by $[0,1]^d$ is no longer set apart as search space. In consequence and to avoid the complications arising with the polytope $A([0,1]^d)$ we will therefore focus on functions defined on the Euclidean ball.\\
To be precise, we consider functions $f:B_{\mathbb R^d}(1+\bar\epsilon) \to \mathbb R$ of the form \eqref{model1}, where $A$ is an arbitrary $k \times d$ matrix whose rows are in $\ell_q^d$, for some $0< q \leq 1$,
$$
\left (\sum_{j=1}^d |a_{ij}|^q \right )^{1/q} \leq C_1.
$$
Further, we assume, that the function $g$ is defined on the image of 
$B_{\mathbb R^d}(1+\bar\epsilon)$ under the matrix $A$ and is twice continuously differentiable on this domain,
i.e., $g\in C^2(AB_{\mathbb R^d}(1+\bar\epsilon))$, and
$$
\max_{|\alpha|\le 2}\|D^\alpha g\|_\infty\le C_2.
$$
For $\mu_{\mathbb S^{d-1}}$ the uniform surface measure on the sphere $\mathbb S^{d-1}$ we define the matrix
\begin{equation}\label{eq:def:hf}
H^f := \int_{\mathbb S^{d-1}} \nabla f(x) \nabla f(x)^T d\mu_{\mathbb S^{d-1}}(x).
\end{equation}
From the identity $\nabla f(x)=A^T \nabla g(Ax)$ we get that
\begin{equation}\label{hfdecomposition}
H^f = A^T \cdot   \int_{\mathbb S^{d-1}} \nabla g(Ax) \nabla g(Ax)^T d\mu_{\mathbb S^{d-1}}(x) \cdot A,
\end{equation}
and therefore that the rank of $H^f$ is $k$ or less. We will require $H^f$ to be well conditioned, i.e., that its singular values satisfy $\sigma_1(H^f) \geq \dots \geq \sigma_k(H^f) \geq \alpha >0$.\\
The parameters in our model are the dimension $d$ (large), the linear parameter dimension $k$ (small), the nonnegative constants $C_1,C_2$, $0< q \leq 1$, and $0 < \alpha \leq kC_2^2$.\\
We now show that such a model can be simplified as follows. 
First of all we see that giving up the orthonormality condition on the rows of $A$ was actually unnecessary. Let us consider the singular value decomposition of $A = U \Sigma V^T$, hence we rewrite
$$
f(x) = g( A x) = \tilde g( \tilde A x), \quad \tilde A \tilde A^T = I_k,
$$
where $\tilde g(y) = g( U \Sigma y)$ and $\tilde A = V^T$. In particular, by simple direct computations,
\begin{itemize}
\item $\sup_{|\alpha|\leq 2} \| D^\alpha \tilde g\|_\infty \leq  \sup_{|\alpha|\leq 2} \| D^\alpha g\|_\infty \cdot \max\{\sqrt{k} \sigma_1(A), k \sigma_1(A)^2\}$, and
\item $\left (\sum_{j=1}^d |\tilde a_{ij}|^q \right )^{1/q} \leq  C_1\sigma_k(A)^{-1}k^{1/q-1/2}$.
\end{itemize}
Hence, by possibly considering different constants $\tilde C_1 = k^{1/q-1/2}  \sigma_k(A)^{-1} C_1$ and $\tilde C_2=\max\{\sqrt{k} \sigma_1(A), k \sigma_1(A)^2\} C_2$, 
we can always assume that $A A^T = I_k$, meaning $A$ is row-orthonormal. Note that for a row-orthonormal matrix $A$, 
equation \eqref{hfdecomposition} tells us that the singular values of $H^f$ are the same as those of $H_g$, where 
$$
H_g:=\int_{\mathbb S^{d-1}} \nabla g(Ax) \nabla g(Ax)^T d\mu_{\mathbb S^{d-1}}(x).
$$
The following simple result states that our model is almost well-defined. As we will see later, the conditions on $A$ and $f$ will be sufficient for the unique identification of $f$ by approximation up to any accuracy, but not necessarily for the unique identification of $A$ and $g$.
\begin{lemma}
Assume that $f(x) = g(A x) = \tilde g(\tilde A x)$ with $A, \tilde A$ two $k\times d$ matrices such that $A A^T = I_k = \tilde A \tilde A^T$ and that $H^f$ has rank $k$. Then $\tilde A = \mathcal O A$ for some $k \times k$ orthonormal matrix $\mathcal O$.
\end{lemma}
\begin{proof}
Because $A$ and $\tilde A$ are row-orthonormal the singular values of $H_g$ and $H_{\tilde g}$ are the same as those of $H^f$, i.e., we have $H_g = U \Sigma U^T$ and $H_{\tilde g} = \tilde U \Sigma \tilde U^T$, where $\Sigma$ is a $k\times k$ diagonal matrix containing the singular values of $H^f$ in nonincreasing order and $U, \tilde U$ are orthonormal $k\times k$ matrices. Inserting this into \eqref{hfdecomposition} we get
\begin{align*}
 H^f &=  A^T H_g A =  A^T  U \Sigma U^T A\\
  &=  \tilde A^T H_{\tilde g} \tilde A= \tilde A^T \tilde U \Sigma \tilde U^T \tilde A.
\end{align*}
$U^TA$ and $\tilde U^T \tilde A$ are both row-orthonormal, so we have two singular value decompositions of $H^f$. Because the singular vectors are unique up to an orthonormal transform, we have $\tilde U^T \tilde A = V U^TA$ for some orthonormal matrix $V$ or $\tilde A = \mathcal O A$ for $\mathcal O =  \tilde U V U^T$, which is by construction orthonormal.
\end{proof}
With the above observations in mind, let us now restate the problem we are addressing and summarize our requirements. We restrict the learning problem to functions 
$f:B_{\mathbb \R^d}(1+\bar \epsilon) \to \mathbb R$ of the form $f(x) = g(A x)$, 
where $A \in M_{k \times d}$ and $A A^T = I_k$. As we are interested in recovering $f$ 
from a small number of samples, the accuracy will depend on the smoothness of $g$. In order to get simple convergence estimates, 
we require $g \in C^2(B_{\mathbb \R^k}(1+\bar \epsilon))$. These choices determine two positive constants $C_1, C_2$ for which
\begin{equation}
\label{cond1}
\left (\sum_{j=1}^d |a_{ij}|^q \right )^{1/q} \leq C_1,
\end{equation}
and
\begin{equation}
\label{cond2}
\sup_{|\alpha|\leq 2} \| D^\alpha g\|_\infty \leq C_2. 
\end{equation}
For the problem to be well-conditioned we need that the matrix $H^f$ is positive definite
\begin{equation}
\label{cond3}
\sigma_1(H^f) \geq \dots \geq \sigma_k(H^f) \geq \alpha,
\end{equation}
for a fixed constant $\alpha >0$ (actually later we may simply choose $\alpha= \sigma_k(H^f)$).

\begin{remark}\label{conditionref} Let us shortly comment on condition \eqref{cond3} in the most simple case $k=1$, by showing that such a condition is actually necessary in order to formulate a \emph{tractable} algorithm for the uniform approximation of $f$ from point evaluations. 
\\
The optimal choice of $\alpha$ is given by
\begin{equation}\label{model:alpha}
\alpha=\int_{{\mathbb S}^{d-1}}|g'(a\cdot x)|^2 d\mu_{{\mathbb S}^{d-1}}(x)=\frac{\Gamma(d/2)}{\pi^{1/2}\Gamma((d-1)/2)}\int_{-1}^1(1-|y|^2)^{\frac{d-3}{2}}dy,
\end{equation}
cf. Theorem \ref{thmpush}. Furthermore, we consider the function $g\in C^2([-1-\bar\epsilon,1+\bar\epsilon])$
given by $g(y)=8(y-1/2)^3$ for $y\in[1/2,1+\bar\epsilon]$ and zero otherwise. Notice that, for every $a\in\R^d$ with $\|a\|_{\ell_2^d}=1$, the function $f(x)=g(a\cdot x)$
vanishes everywhere on ${\mathbb S}^{d-1}$ outside of the cap ${\mathcal U}(a,1/2):=\{x\in{\mathbb S}^{d-1}:a\cdot x\ge 1/2\}$, see Figure \ref{gcap}. The $\mu_{{\mathbb S}^{d-1}}$ measure of ${\mathcal U}(a,1/2)$ obviously does not depend on $a$ and is known to be exponentially small in $d$ \cite{le01}, see also Section \ref{sec:tract}. Furthermore, it is known, 
that there is a constant $c>0$ and unit vectors $a^1,\dots,a^K$, such that the sets ${\mathcal U}(a^1,1/2),\dots,{\mathcal U}(a^K,1/2)$ are mutually disjoint and $K\ge e^{cd}$.
Finally, we observe that $\max_{x\in{\mathbb S}^{d-1}}|f(x)|=f(a)=g(1)=1.$

\begin{figure}[h]
\hfill {\includegraphics[width=4.5cm]{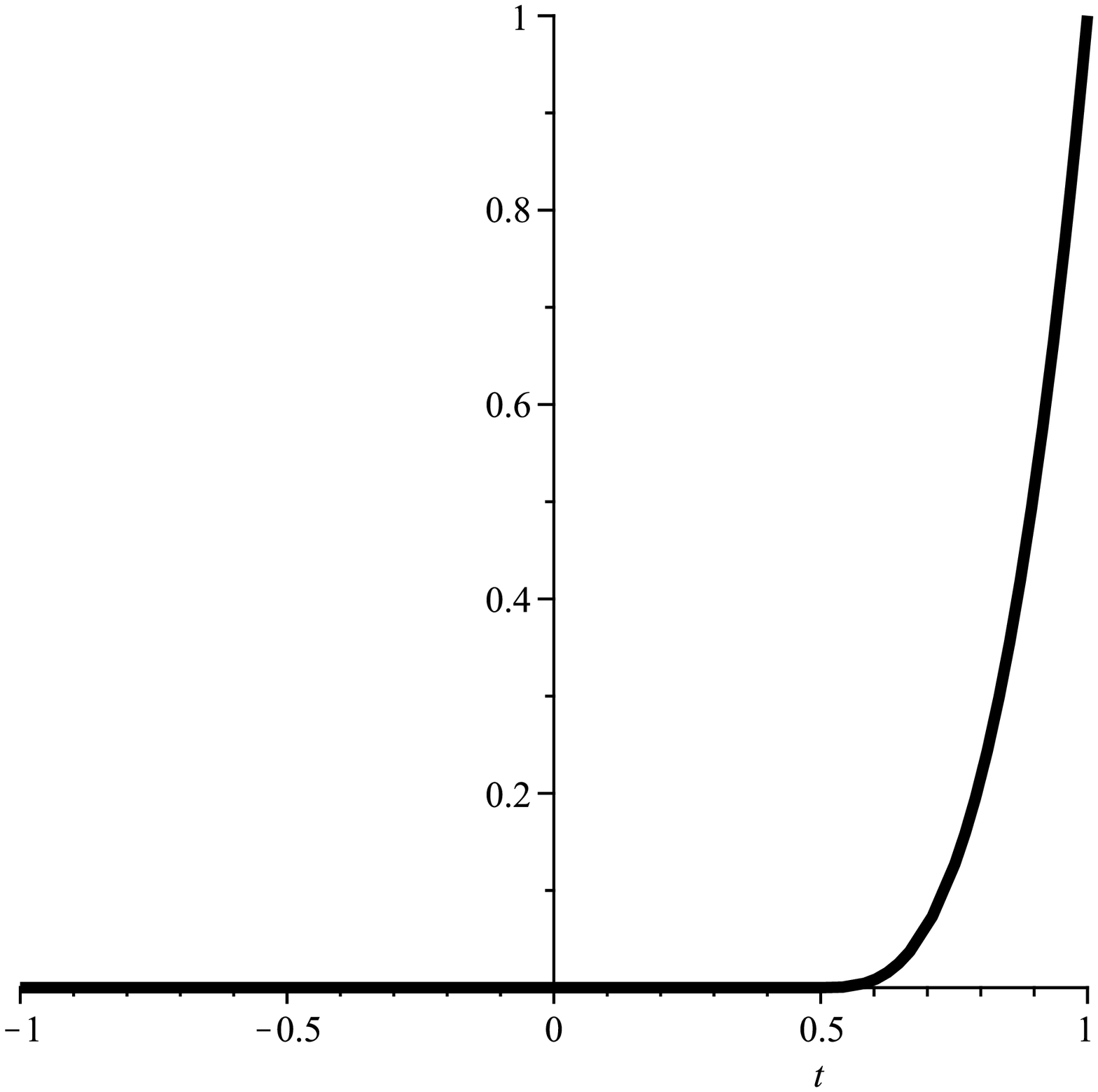}}
\hfill {\includegraphics[width=4.5cm]{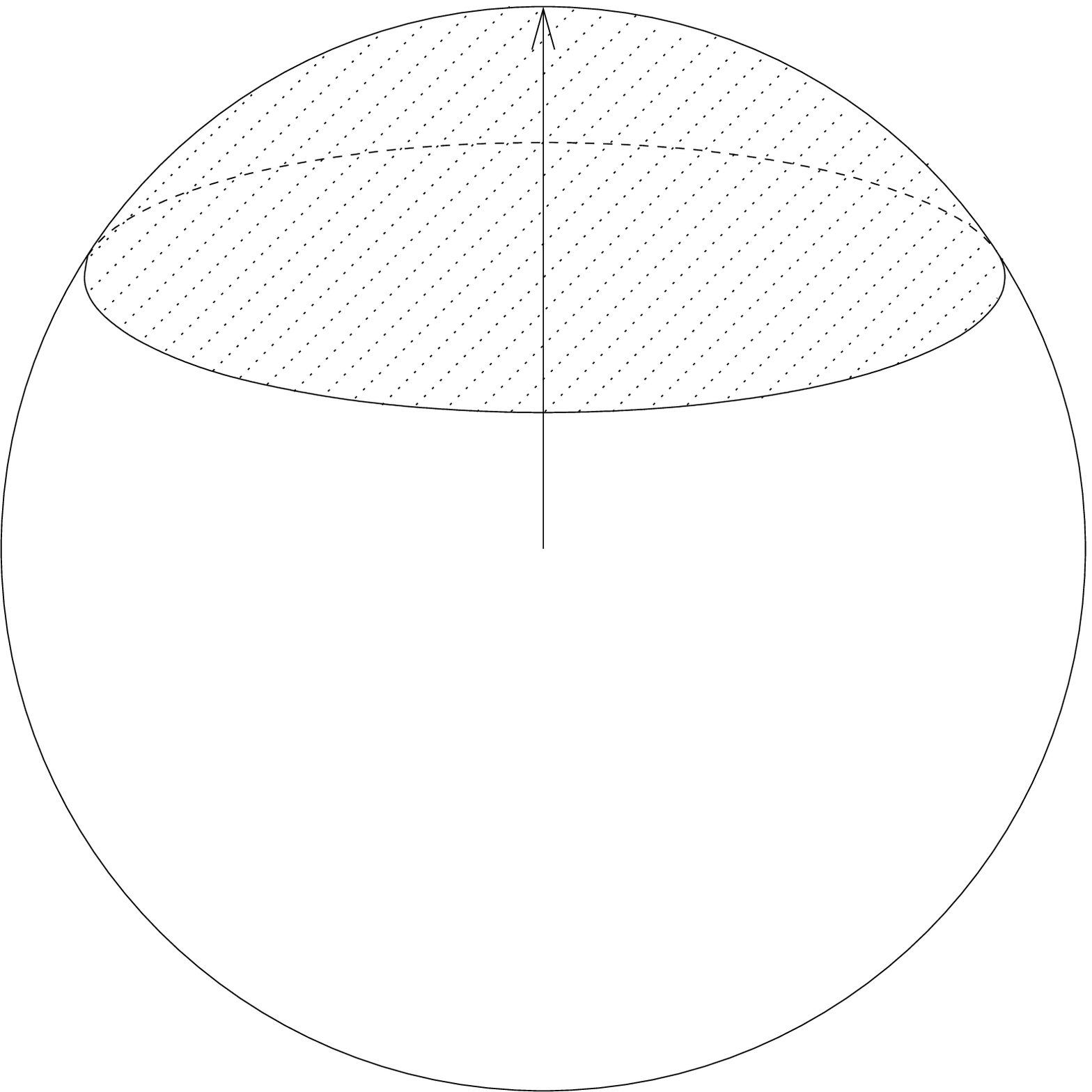}}\hfill\ \\
\caption{The function $g$ and the spherical cap ${\mathcal U}(a,1/2)$.}
\label{gcap}
\end{figure}

We conclude that \emph{any} algorithm making only use of the structure of $f(x)=g(a\cdot x)$ and the condition \eqref{cond2} needs to use exponentially many sampling points
in order to distinguish between $f(x) \equiv 0$ and $f(x)=g(a^i\cdot x)$ for some of the $a^i$'s as constructed above. Hence, some additional conditions like \eqref{cond1} and 
\eqref{cond3} are actually necessary to avoid the curse of dimensionality and to achieve at least some sort of tractability. Let us observe that $\alpha=\alpha(d)$
decays exponentially with $d$ for the function $g$ considered above. We shall further discuss the role of $\alpha$ in Section \ref{sec:tract}.
\end{remark}


Contrary to the approach in \cite{codadekepi10} our strategy used to learn functions of the type \eqref{model1} is to first find an approximation $\hat{A}$ to $A$. 
Once this is known, we will give a pointwise definition of the function $\hat{g}$ on $B_{\mathbb \R^k}(1)$ such that $\hat f (x) := \hat g(\hat A x)$ is a good 
approximation to $f$ on $B_{\mathbb \R^d}(1)$. This will be in a way such that the evaluation of $\hat g$ at one point will require only one function evaluation 
of $f$. Consequently, an approximation of $\hat{g}$ on its domain $B_{\mathbb \R^k}(1)$ using standard techniques, like sampling on a regular grid and spline-type 
approximations, will require a number of function evaluations of $f$ depending only on the desired accuracy and $k$, but not on $d$. We will therefore restrict 
our analysis to the problem of finding $\hat A$, defining $\hat g$, and the amount of queries necessary to do that.

\section{The One Dimensional Case $k=1$}\label{sec:kis1}
For the sake of an easy introduction, we start by addressing our recovery method again in the simplest case of a {\it ridge function}
\begin{equation}
\label{ridge}
f(x) = g(a \cdot x), 
\end{equation}
where $a=(a_1,\dots,a_d)\in \mathbb R^d$ is a row vector, $\| a\|_{\ell_2^d}=1$, and 
$g$ is a function from the image of $B_{\mathbb R^d}(1+\bar \epsilon)$ 
under $a$ to $\mathbb R$, i.e., $g: B_{\mathbb \R}(1+\bar \epsilon) \to \mathbb R$.
\\
The ridge function terminology was introduced in the 1970's by Logan and Shepp \cite{losh75} in connection with the
mathematics of computer tomography. However these functions have been considered for some time, but under the name of 
\emph{plane waves}. See, for example, \cite{cohi62,jo55}.
Ridge functions and ridge function approximation are studied in statistics. There they often go under the name of projection pursuit. Projection pursuit algorithms approximate a function of $d$ variables by functions of the form
\begin{equation}\label{kridge'}
f(x) \approx \sum_{j=1}^\ell g_j(a_j \cdot x).
\end{equation}
Hence the recovery of $f$ in \eqref{ridge} from few samples can be seen as an instance of the projection pursuit problem. For a survey on some approximation-theoretic questions concerning ridge functions and their connections to neural networks, see \cite{pi99} and references therein, 
and the work of Cand\`es and Donoho on ridgelet approximation \cite{ca99,ca03,cado99}.\\
For further clarity of notations, in the following we will assume $a$ to be a row vector, i.e., a $1 \times d$ matrix, while other vectors, $x, \xi, \varphi \dots$, are always assumed to be column vectors. Hence the symbol $a \cdot x$ stands for the product of the $1 \times d$ matrix $a$ with the $d \times 1$ vector $x$.

\subsection{The Algorithm}\label{sec:algkis1}
As in \cite{codadekepi10} a basic ingredient of the algorithm is a version of Taylor's theorem giving access to the vector $a$.
For $\xi \in B_{\mathbb \R^d}$, $\varphi \in B_{\mathbb R^d}(r)$, $\epsilon,r \in \mathbb R_+$, with $r \epsilon \leq \bar \epsilon$, 
we have, by Taylor expansion, the identity
\begin{eqnarray}
[g'(a \cdot \xi) a ]\cdot \varphi &=& \frac{\partial f}{\partial \varphi} (\xi) \nonumber\\
&=& \frac{f(\xi + \epsilon \varphi) - f(\xi)}{\epsilon} - \frac{\epsilon}{2} [\varphi^T \nabla^2 f(\zeta) \varphi]\label{taylor},
\end{eqnarray}
for a suitable $\zeta(\xi,¸\varphi) \in B_{\mathbb R^d}(1+\bar \epsilon)$.
Thanks to our assumptions \eqref{cond1} and \eqref{cond2}, the term $[\varphi^T \nabla^2 f(\zeta) \varphi]$ 
is uniformly bounded as soon as $\varphi$ is bounded. 
We will consider the above equality for several directions $\varphi_i$ and at several sampling points $\xi_j$.



To be more precise we define two sets $\mathcal X, \Phi$ of points. The first 
\begin{equation}
\label{Xpoints}
\mathcal X = \{ \xi_j \in \mathbb S^{d-1}: j=1,\dots, m_{\mathcal X}\},
\end{equation}
contains the $m_{\mathcal X}$ sampling points and is drawn at random in $\mathbb S^{d-1}$ according to the probability measure $\mu_{\mathbb S^{d-1}}$. For the second, containing the $m_{\Phi}$ derivative  directions, we have
\begin{eqnarray}
\Phi &=& \left \{ \varphi_i \in B_{\mathbb R^d}(\sqrt d/\sqrt{m_\Phi}): 
\varphi_{i\ell} = \frac{1}{\sqrt m_\Phi}
\left \{ 
\begin{array}{ll}
1, & \mbox{ with probability 1/2}, \nonumber \\
-1, & \mbox{ with probability 1/2},
\end{array}
\right . \right.\\
&& \left . \phantom{XXXXXXXXXXXXXXXXX} i=1,\dots, m_\Phi, \mbox{ and } \ell=1,\dots,d  \right \}. \label{Phipoints}
\end{eqnarray}
Actually we identify $\Phi$ with the $m_\Phi \times d$ matrix whose rows are the vectors $\varphi_i$. To write the $m_{\mathcal X} \times m_\Phi$ instances of~\eqref{taylor} in a concise way we collect the directional derivatives $g'(a \cdot \xi_j) a$, $j=1,\dots, m_{\mathcal X}$ as columns in the $d \times m_{\mathcal X}$ matrix $X$, i.e.,  
\begin{equation}
X=(g'(a \cdot \xi_1) a^T,\ldots, g'(a \cdot \xi_{m_\mathcal{X}}) a^T),
\end{equation}
and we define the $m_\Phi \times m_{\mathcal X}$ matrices $Y$ and $\mathcal E$ entrywise by
\begin{equation}
\label{yij}
y_{ij} = \frac{f(\xi_j + \epsilon \varphi_i) - f(\xi_j)}{\epsilon},
\end{equation}
and
\begin{equation}
\label{eij}
\varepsilon_{ij} =   \frac{\epsilon}{2} [\varphi^T_i \nabla^2 f(\zeta_{ij}) \varphi_i].
\end{equation}
We denote by $y_j$ the columns of $Y$ and by $\varepsilon_j$ the columns of $\mathcal E$, $j=1,\dots,m_{\mathcal X}$.
With these matrices we can write the following factorization
\begin{equation}
\label{factor0}
\Phi X = Y - \mathcal E.
\end{equation}
The algorithm we propose to approximate the vector $a$ is now based on the fact that the matrix $X$ has a very special structure,
i.e., $X= a^T \mathcal G^T$,
where $\mathcal G = ( g'(a \cdot \xi_1), \dots, g'(a \cdot \xi_{m_\mathcal{X}}))^T$. In other words every column $x_j$ is a scaled copy of the vector $a^T$ and {\it compressible} if $a$ is {\it compressible}. 
We define a vector $a$ compressible informally by saying that it can be well approximated in $\ell_p$-norm by a sparse vector. Actually, any vector $a$ with small $\ell_q$-norm can be approximated in $\ell_p$ by its best $K$-term approximation $a_{[K]}$ according to the following well-known estimate
\begin{equation}\label{bestkterm}
\sigma_K(x)_{\ell_p^d}:=\| a- a_{[K]}\|_{\ell_p^d} \leq \|a\|_{\ell_q^d} K^{1/p-1/q}, \quad p\geq q.
\end{equation}
Thus by changing view point to get
$$
Y=\Phi X + \mathcal E
$$
we see that due to the random construction of $\Phi$ we actually have a {\it compressed sensing} problem and known theory tells us that we can recover a stable 
approximation $\hat x_j$ to $x_j$ via $\ell_1$-minimization (see Theorem \ref{thm:cs} for the precise statement). To get an approximation of $a$ we then simply 
have to set $\hat a = \hat x_j/\|\hat x_j\|_{\ell_2^d}$ for $j$ such that $\|\hat x_j\|_{\ell_2^d}$ is maximal. 
From these informal ideas we derive the following algorithm.\\

\fbox{
\begin{minipage}{14cm}
\noindent{\bf Algorithm 1}:
{\it
\begin{itemize}
\item Given  $m_\Phi, m_{\mathcal X}$, draw at random the sets $\Phi$ and $\mathcal X$ as in \eqref{Xpoints} and \eqref{Phipoints}, and construct $Y$ according to~\eqref{yij}.
\item Set $\hat x_j = \Delta (y_j) := \argmin_{y_j=\Phi z } \| z\|_{\ell_1^d}$.
\item Find 
\begin{equation}\label{maxchoice}
j_0=\arg\max_{j=1, \ldots, m_{\mathcal{X}}} \|\hat x_j\|_{\ell_2^d}.
\end{equation}
\item Set $\hat a = \hat x_{j_0}/\|\hat x_{j_0}\|_{\ell_2^d}$.
\item Define $\hat g (y):= f(\hat a^T y)$ and $\hat f (x): = \hat g (\hat a \cdot x)$.
\end{itemize}
}
\end{minipage}
}
\\

The quality of the final approximation clearly depends on the error between $\hat x_j$ and $x_j$, which can be controlled 
through the number of {\it compressed sensing measurements} $m_\Phi$, and the size of $\hat a \approx \max_j \|x_j\|_{\ell_2^d} 
= \max_j |g'(a\cdot \xi_j)|$, which is related to the number of random samples $m_\mathcal{X}$. If $\eqref{model:alpha}$ is satisfied with
$\alpha$ large, we shall show in Lemma \ref{lem2} with help of Hoeffding's inequality that also $\max_j \|x_j\|_{\ell_2^d}=\max_j |g'(a\cdot \xi_j)|$ 
is large with high probability. If the value of $\alpha$ is unknown and small, the values of $\|\hat x_j\|_{\ell_2^d}$ produced by Algorithm 1
could be small as well and, as discussed after the formula \eqref{model:alpha}, no reliable and tractable approximation procedure is possible.

To be exact we will in the 
next section prove the following approximation result.

\begin{theorem} \label{thm:kis1}
Let $0<s<1$ and $\log d \le m_\Phi\le [\log 6]^{-2} d$. Then there is a constant $c_1'$ such that 
using $m_\mathcal{X} \cdot (m_\Phi+1)$ function evaluations of $f$, Algorithm 1 defines a function 
$\hat f:B_{\mathbb R^d}(1+\bar \epsilon) \to \mathbb R$ that, with probability 
\begin{equation}
1-\left (e^{-c'_1 m_\Phi}+e^{-\sqrt{m_\Phi d}} + 2 e^{-\frac{2 m_{\mathcal X} s^2 \alpha^2}{C_2^4}} \right ),\label{prob:kis1}
\end{equation}
will satisfy
\begin{equation}
\|f -\hat f\|_\infty \leq 2 C_2(1+ \bar \epsilon) \frac{ \nu_1} { \sqrt{\alpha(1-s)}- \nu_1},\label{error:kis1}
\end{equation}
where
\begin{equation}\label{eq:nu_1}
\nu_1=C'\left(\left[\frac{m_\Phi}{\log(d/m_\Phi)}\right]^{1/2-1/q}+\frac{\epsilon}{\sqrt{m_\Phi}}\right)
\end{equation}
and $C'$ depends only on $C_1$ and $C_2$ from \eqref{cond1} and \eqref{cond2}.
\end{theorem}
\begin{remark}
1. We shall fix $\nu_1$ as defined by \eqref{eq:nu_1} for the rest of this section. Furthermore, we suppose that
the selected parameters ($s, \epsilon$ and $m_\Phi$) are such that $\nu_1<\sqrt{\alpha(1-s)}$ holds. See Remark \ref{remark2} (ii)
for knowing how we can circumvent in practice the case that this condition may not hold, clearly invalidating the approximation \eqref{error:kis1}.

 2. In order to show a concrete application of the previous result,  let us consider, for simplicity,  a class of uniformly smooth functions $g$ such that $|g'(0)| \neq 0$; hence, by  Proposition \ref{alphad}, $\alpha=\alpha(g)>0$ is independent of the dimension $d$.
If additionally we choose $q=1$, $m_\Phi < d$, and $\epsilon>0$ such that 
$m_\Phi (\epsilon + \sqrt{ \log(d/m_\Phi)})^{-2}= \mathcal O(\delta^{-2} \alpha^{-1})$, 
$\delta>0$, for $\delta,\alpha \to 0$ and $m_{\mathcal X} = \mathcal O(\alpha^{-2})$ for $\alpha \to 0$, then, according to Theorem \ref{thm:kis1},
we obtain the uniform error estimate
$$
\|f -\hat f\|_\infty = \mathcal O\left ( \delta\right), \quad \delta \to 0,
$$
with high probability. Notice that, if $1/\log(d)>\delta>0$, then the number of evaluation points 
$m_{\mathcal X}\cdot (m_\Phi +1) =\mathcal O( (\delta \cdot \alpha)^{-3})$,
for $\delta,\alpha \to 0$, is actually independent of the dimension $d$.

\end{remark}
\subsection{The Analysis}
We will first show that $\hat x_j$ is a good approximation to $x_j$ for all $j$. This follows by the results from the framework of {\it compressed sensing} \cite{badadewa08,
carota06-1,codade09,do06-2,fo10,fo09,fora10}. In particular, we state the following useful result which is a specialization of Theorem 1.2 from \cite{woxx},
to the case of Bernoulli matrices. 
\begin{theorem}\label{thm:cs}
Assume that $\Phi$ is an $m\times d$ random matrix with all entries being independent Bernoulli variables scaled with $1/\sqrt{m}$, see, e.g., \eqref{Phipoints}.

(i) Let $0<\delta<1$. Then there are two positive constants $c_1,c_2>0$, such that the matrix $\Phi$ has the Restricted Isometry Property
\begin{equation}\label{RIP}
(1- \delta) \| x \|_{\ell_2^d}^2 \leq \| \Phi x \|_{\ell_2^m}^2\leq  (1+ \delta) \| x \|_{\ell_2^d}^2
\end{equation}
for all $x \in \mathbb R^d$ such that $\# \supp(x) \leq c_2 m/\log(d/m)$ with probability at least
\begin{equation}
1-e^{-c_1 m}.
\end{equation}

(ii) Let us suppose that $d>[\log 6]^2 m$. Then there are positive constants $C,c_1', c_2'>0$, such that, with probability at least
\begin{equation}\label{woj2}
1-e^{-c'_1 m}-e^{-\sqrt{md}},
\end{equation}
the matrix $\Phi$ has the following property. For every $x\in\R^d$, $\varepsilon \in \R^m$ and every natural number 
$K\le c_2'm/\log(d/m)$ we have
\begin{equation}\label{woj1}
\| \Delta( \Phi x + \varepsilon ) - x\|_{\ell_2^d} \leq 
C\left( K^{-1/2}\sigma_K(x)_{\ell_1^d} + \max\{\|\varepsilon \|_{\ell_2^m},\sqrt{\log d}\|\varepsilon \|_{\ell_\infty^m} \}  \right),
\end{equation}
where 
$$
\sigma_K(x)_{\ell_1^d}:=\inf\{\|x-z\|_{\ell_1^d}:\#\supp z\le K\}
$$
is the \emph{best $K$-term approximation of $x$.}
\end{theorem}

\begin{remark}
\noindent (i) The first part of Theorem \ref{thm:cs} is well known, see, e.g., \cite{badadewa08} or \cite[Page 15]{fo10}
and references therein.\\
\noindent (ii) The second part of Theorem \ref{thm:cs} is relatively new. It follows from Theorem 2.3 of \cite{woxx}
combined with Theorem 3.5 of \cite{depewo09}, and the first part of Theorem \ref{thm:cs}.
Without the explicit bound of the probability \eqref{woj2}, it appears also as Theorem 1.2 in \cite{woxx}.
\end{remark}

\noindent Applied to the situation at hand we immediately derive the following corollary.

\begin{corollary}\label{cor1}
(i) Let $d>[\log 6]^2 m_\Phi$. Then with probability at least
$$
1-(e^{-c'_1 m_\Phi}+e^{-\sqrt{m_\Phi d}})
$$ 
all the vectors $\hat x_j = \Delta(y_j),\, j=1,\dots, m_{\mathcal X}$ calculated in Algorithm 1 satisfy
\begin{equation}\label{approx2}
\| x_j- \hat x_j\|_{\ell_2^d} \leq 
C\left(\left[\frac{m_\Phi}{\log(d/m_\Phi)}\right]^{1/2-1/q}+\max\{\|\varepsilon_j\|_{\ell_2^{m_\Phi}},\sqrt{\log d}\|\varepsilon_j \|_{\ell_\infty^{m_\Phi}} \}\right)
\end{equation}
where $C$ depends only on $C_1$ and $C_2$ from \eqref{cond1} and \eqref{cond2}.

(ii) If furthermore $m_\Phi\ge \log d$ holds, then with the same probability also
\begin{equation}\label{approx2-1}
\| x_j- \hat x_j\|_{\ell_2^d} \leq 
C'\left(\left[\frac{m_\Phi}{\log(d/m_\Phi)}\right]^{1/2-1/q}+\frac{\epsilon}{\sqrt{m_\Phi}}\right)
\end{equation}
where $C'$ depends again only on $C_1$ and $C_2$ from \eqref{cond1} and \eqref{cond2}.
\end{corollary}

\begin{proof} We apply Theorem \ref{thm:cs} to the equation $y_j=\Phi x_j+\varepsilon_j$ and $K\leq c_2' m_\Phi/\log(d/m_\Phi)$.
To do so, we have to estimate the best $K$-term approximation error of $\sigma_K(x_j)_{\ell_1^d}$ and the size of the errors $\varepsilon_j$.
We start by bounding $\sigma_K(x_j)_{\ell_1^d}$. Recall that due to the construction of $X$ every column is a scaled copy of the vector $a^T$,
i.e., $x_j = g'(a \cdot \xi_j) a^T$, so we have by \eqref{bestkterm}
\begin{equation}
K^{-1/2}\sigma_K(x_j)_{\ell_1^d}\le |g'(a \cdot \xi_j)|\cdot\|a\|_{\ell_q^{d}}\cdot K^{1/2-1/q}\le C_1\,C_2\left[\frac{m_\Phi}{\log(d/m_\Phi)}\right]^{1/2-1/q}.
\end{equation}
This finishes the proof of the first part.

To prove the second part, we estimate the size of the errors using~\eqref{eij},
\begin{eqnarray}
\notag \|\varepsilon_j\|_{\ell_\infty^{m_\Phi}}&=&\frac{\epsilon}{2}\cdot \max_{i=1,\ldots, m_\Phi}|\varphi_i^T \nabla^2f(\zeta_{ij})\varphi_i|\\
\label{eq:este1}&=&\frac{\epsilon}{2m_\Phi}\cdot\max_{i=1,\ldots, m_\Phi}\left|\sum_{k,l=1}^d a_ka_lg''(a\cdot\zeta_{ij})\right| \\
\notag&\le&\frac{\epsilon\|g''\|_\infty}{2m_{\Phi}} \left(\sum_{k=1}^d |a_k|\right)^2
\le\frac{\epsilon\|g''\|_\infty}{2m_{\Phi}}\left(\sum_{k=1}^d |a_k|^q\right)^{2/q}
\le \frac{C_1^2 C_2}{2m_{\Phi}}\epsilon,\\
\label{eq:este2}\|\varepsilon_j\|_{\ell_2^{m_\Phi}}&\leq &\sqrt{m_\Phi} \|\varepsilon_j\|_{\ell_\infty^{m_\Phi}} \le \frac{C_1^2 C_2}{2\sqrt{m_{\Phi}}}\epsilon,
\end{eqnarray}
leading to 
$$
\max\{\|\varepsilon_j \|_{\ell_2^{m_\Phi}},\sqrt{\log d}\|\varepsilon_j \|_{\ell_\infty^{m_\Phi}} \} \leq  \frac{C_1^2 C_2}{2\sqrt{m_{\Phi}}}\epsilon \cdot \max \left \{1, \sqrt{\frac{\log d}{m_\Phi} } \right \}.
$$
Together with our assumption $m_\Phi\ge \log d$ this finishes the proof.
\end{proof}

\noindent Next we need a technical lemma to relate the error between the normalized version of $\hat x_j$ and $a$ to the size of $\| \hat x_j\|_{\ell_2^d}$.\\

\begin{lemma}[Stability of subspaces - one dimensional case]
\label{lem1}
Let us fix $\hat x \in \mathbb R^d$, $a \in \mathbb S^{d-1}$, $0 \neq \gamma \in \mathbb R$, and $n \in \mathbb R^d$ with norm 
$\| n \|_{\ell_2^d} \leq \nu_1 <|\gamma|$. If we assume $\hat x = \gamma a + n $ 
then
\begin{eqnarray}
\left \| \operatorname{sign}{\gamma}\frac{\hat x}{\|\hat x\|_{\ell_2^d}} - a \right \|_{\ell_2^d}
&\leq&  \frac{2 \nu_1} {\|\hat x\|_{\ell_2^d} }\label{tricky}.
\end{eqnarray}
\end{lemma}
\begin{proof}
Applying the triangular inequality and its reverse form several times and using that $a \in \mathbb S^{d-1}$ we get
\begin{align*}
\left \| \operatorname{sign}{\gamma}\frac{\hat x}{\|\hat x\|_{\ell_2^d}} - a \right \|_{\ell_2^d} &\leq \left \| \operatorname{sign}{\gamma}\frac{\hat x}{\|\hat x\|_{\ell_2^d}} -\frac{|\gamma| a}{\|\hat x\|_{\ell_2^d}} \right \|_{\ell_2^d} + \left \|  \frac{|\gamma| a}{\|\hat x\|_{\ell_2^d}} - a \right \|_{\ell_2^d} \\
&\leq \frac{\nu_1}{\|\hat x\|_{\ell_2^d}} + \left|  \frac{|\gamma|}{\|\hat x\|_{\ell_2^d}} - 1 \right |\leq \frac{2 \nu_1}{\|\hat x\|_{\ell_2^d}}.
\end{align*}
\end{proof}

Applied to our situation where $\hat x_j = g'(a \cdot \xi_j) a^T + n_j$ we see that the bound in~\eqref{tricky} is best for $\|\hat x_j\|_{\ell_2^d}$ maximal which justifies our definition of $\hat a$ in Algorithm 1.\\

As a last ingredient for the proof of Theorem~\ref{thm:kis1} we need a lower bound for $\max_{j=1,\dots,m_{\mathcal X}} \|\hat x\|_{\ell_2^d}$.
Since we have $\max_j \|\hat x_j\|_{\ell_2^d} \geq \max_j |g'(a \cdot \xi_j)| - \max_{j}\|\hat x_j-x_j\|_{\ell_2^d}
\geq \max_j |g'(a \cdot \xi_j)|-\nu_1$ we just have to show that, with 
high probability, our random sampling of the gradient via the $\xi_j$ provided a good maximum. To do this we will use 
Hoeffding's inequality, which we recall below for reader's convenience.
\begin{prop}[Hoeffding's inequality]
\label{hoeffprop}
Let $X_1, \dots, X_m$ be independent random variables. Assume that the $X_j$ are almost surely bounded, i.e., there exist finite scalars $a_j, b_j$ such that
$$
\mathbb P \{ X_ j - \mathbb E X_j \in [a_j, b_j] \} =1,
$$
for $j=1,\dots,m$. Then we have
$$
\mathbb P \left \{ \left | \sum_{j=1}^m X_j - \mathbb E \left ( \sum_{j=1}^m X_j \right) \right | \geq t\right \} \leq 2 e^{-\frac{2 t^2}{\sum_{j=1}^m (b_j-a_j)^2}}.
$$
\end{prop}
\noindent Let us now apply Hoeffding's inequality to the random variables $X_j = | g'(a \cdot \xi_j)|^2$. 
\begin{lemma}
\label{lem2}
Let us fix $0<s<1$. Then with probability $1 - 2 e^{-\frac{2 m_{\mathcal X} s^2 \alpha^2}{C_2^4}}$ we have
$$
\max_{j=1,\ldots, m_\mathcal{X}} | g'(a \cdot \xi_j)| \geq \sqrt{\alpha(1-s)},
$$
where $\alpha := \mathbb E_\xi (| g'(a \cdot \xi_j)|^2)$.
\end{lemma}
\begin{proof}
By our assumptions \eqref{cond3} and \eqref{cond2} we have 
$$
\mathbb E X_j =\mathbb E_\xi (| g'(a \cdot \xi_j)|^2) = \int_{\mathbb S^{d-1}} | g'(a \cdot \xi)|^2 d \mu_{\mathbb S^{d-1}}(\xi) \ge \alpha>0,
$$
and
$$
X_ j - \mathbb E X_j \in [- \alpha , C_2^2-\alpha].
$$
Hence, by Hoeffding's inequality we have 
\begin{equation}
\label{hoeff}
\mathbb P \left \{ \left | \sum_{j=1}^{m_{\mathcal X}}  | g'(a \cdot \xi_j)|^2 - m_{\mathcal X} \alpha \right | \geq s m_{\mathcal X} \alpha \right \} \leq 2 e^{-\frac{2 m_{\mathcal X} s^2 \alpha^2}{C_2^4}}.
\end{equation}
Using \eqref{hoeff} we immediately obtain
\begin{equation}
\label{res}
\frac{1}{m_{\mathcal X}} \sum_{j=1}^{m_{\mathcal X}}  | g'(a \cdot \xi_j)|^2 \geq \alpha(1 -s),
\end{equation}
with probability $1 - 2 e^{-\frac{2 m_{\mathcal X} s^2 \alpha^2}{C_2^4}}$. If $| g'(a \cdot \xi_j)|^2 < \alpha(1-s)$ for all $j=1, \dots, m_{\mathcal X}$ then \eqref{res} would be violated. Hence for the maximum we have 
$$
\max_{j=1,\ldots, m_\mathcal{X}} | g'(a \cdot \xi_j)| \geq \sqrt{\alpha(1-s)}.
$$
\end{proof}

\noindent Finally we have all the tools ready to prove Theorem~\ref{thm:kis1}.\\

\noindent{\bf Proof of Theorem~\ref{thm:kis1}:}

\begin{proof}
Lemma \ref{lem2} ensures that
$$
|g'(a\cdot\xi_{j_0})|\ge \sqrt{\alpha(1-s)}
$$
with probability $1 - 2 e^{-\frac{2 m_{\mathcal X} s^2 \alpha^2}{C_2^4}}$.
Therefore, Corollary~\ref{cor1} together with Lemma~\ref{lem1} show that with probability at least
$$
1-\left (e^{-c'_1 m_\Phi}+e^{-\sqrt{m_\Phi d}}+ 2 e^{-\frac{2 m_{\mathcal X} s^2 \alpha^2}{C_2^4}} \right ),
$$
$\hat a$ as defined in Algorithm~1 satisfies
\begin{equation}\label{approx2'}
\left \| \operatorname{sign}({g'(a \cdot \xi_{j_0})})\hat a - a \right \|_{\ell_2^d}
\leq  \frac{2 \nu_1} { \sqrt{\alpha(1-s)}- \nu_1}
\end{equation}
for the unknown sign of $g'(a \cdot \xi_{j_0})$.\\
Using this estimate we can prove that $\hat f$ as defined in Algorithm 1 is a good approximation to $f$.
For $x \in B_{\mathbb \R^d}(1+ \bar \epsilon)$ we have,
\begin{eqnarray*}
|f(x) - \hat f (x)|&=& |g (a\cdot x) - \hat g (\hat a \cdot x) | \\
&=& |g (a\cdot x) - f(\hat a^T\cdot \hat a\cdot x) | \\
&=& |g (a\cdot x) - g(a \cdot \hat a^T\cdot\hat a \cdot x) |\\
&\leq& C_2 |  a \cdot x - a \cdot [\hat a^T \hat a] \cdot x|\\
&=& C_2|a\cdot (I_d-\hat a^T\hat a)x|.
\end{eqnarray*}
Because $\hat a (I_d-\hat a^T\hat a)=0$ and therefore $\operatorname{sign}(g'(a \cdot \xi_{j_0})) \hat a (I_d-\hat a^T\hat a)=0$, we can further estimate
\begin{eqnarray*}
|f(x) - \hat f (x)| &\leq & C_2|a\cdot (I_d-\hat a^T\hat a)x|\\
&=& C_2|(a-\operatorname{sign}(g'(a \cdot \xi_{j_0}))\hat a)\cdot (I_d-\hat a^T\hat a)x|\\
&\leq&C_2 \|a-\operatorname{sign}(g'(a \cdot \xi_{j_0}))\hat a\|_{\ell_2^d}\cdot \|x\|_{\ell_2^d}\\
&\leq&2C_2(1+ \bar \epsilon) \frac{ \nu_1} { \sqrt{\alpha(1-s)}-\nu_1}.
\end{eqnarray*}
\end{proof}

\begin{remark}\label{remark2}
We collect here a few comments about this result. \\
\noindent (i) Our recovery method differs from the one proposed by Cohen, Daubechies, DeVore, Kerkyacharian, Picard \cite{codadekepi10}. 
In their approach, the domain is taken to be $[0,1]^d$ and they make heavy use of the additional 
assumption $\sum_{j=1}^d a_j=1$ and $a_j \geq 0$. This allows them to derive an almost completely 
deterministic and adaptive strategy for sampling the function $f$ in order to find first an approximation 
to $g$ and only then addressing the approximation to $a$. Here we follow somehow the opposite order, first 
approximating $a$ and then finding a uniform approximation to $g$ and, eventually, to $f$ as well. Notice 
further that not having at disposal additional information on $a$, which is fully arbitrary in our case, 
we need to use a random sampling scheme which eventually gives a result holding with high probability.\\
\noindent (ii) Note that Theorem~\ref{thm:kis1} gives an a priori estimate of the success probability and approximation error of Algorithm 1.
If the problem parameters $q, C_1, C_2$, and $\alpha$ are known, they
can be used to choose $m_\Phi$ and $m_\mathcal{X}$ big enough to have, say, a prescribed desired accuracy $\delta$ with probability
at least $1-p$.\\
However once Algorithm 1 has been run we have the following \emph{a posteriori} estimate. 
With probability at least $1-(e^{-c'_1 m_\Phi}+e^{-\sqrt{m_\Phi d}})$ we have that
$$
\| f - \hat f \|_\infty  \leq C_2(1+ \bar \epsilon) \frac{2 \nu_1} { \|x_{j_0}\|_{\ell_2^d}}.
$$
Hence, the ratio $\frac{2 \nu_1} { \|x_{j_0}\|_{\ell_2^d}} \ll 1$ defines an a posteriori indicator that the number of samples $m_{\mathcal X}$ and $m_{\Phi}$ has been properly calibrated, otherwise just more points will be drawn until such a condition is obtained.\\
\noindent (iii)
The parameter $\epsilon$ is chosen at the very beginning in the Taylor expansion \eqref{taylor} and, from a purely theoretical point 
of view, could be chosen arbitrarily small. Unfortunately, this may affect the numerical stability in the approximation in \eqref{taylor}
of the derivative $\frac{\partial f}{\partial \varphi} (\xi)$ by means of a finite difference.  Hence, the parameter $\epsilon$ 
should not be taken too small in practice. 
Up to some extent this may be compensated by choosing a larger number of points $m_{\Phi}$ in \eqref{eq:nu_1}, as in our expression for $\nu_1$ in \eqref{eq:nu_1}
$\epsilon$ appears in a ratio of the form $\frac{\epsilon}{\sqrt{m_{\Phi}}}$.
We return in more detail to this point later in Section \ref{sec:noise}.
In recent numerical experiments associated to the work \cite{scvy11}, 
we have been experiencing very stable reconstructions with reasonable choices, e.g., $\epsilon \approx 0.1$. Hence we do not consider 
this issue of any practical relevance or difficulty.\\


\end{remark}

\subsection{Discussion on tractability}\label{sec:tract}
The approximation performances of our learning strategy are basically determined by the optimal value of $\alpha$ (see, e.g., \eqref{cond3}),
which is achieved by the choice
\begin{equation}\label{alpha}
\alpha := \int_{\mathbb S^{d-1}} | g'(a \cdot x) |^2 d \mu_{\mathbb S^{d-1}}(x).
\end{equation}
Due to symmetry reasons 
this quantity does not depend on the particular choice of $a$.


The rotation invariant probability measure $\mu_{\mathbb S^{d-1}}$ on $\mathbb S^{d-1}$ is induced on the sphere by the (left)
Haar measure on the Lie group of all orientation preserving rotations. For a given $k \times d$ matrix $U$ such that $U U^T= I_k$
(i.e., with orthonormal rows) we define the measure $\mu_k$ on the unit ball $B_{\mathbb R^k}$ in $\mathbb R^k$ induced by the
projection of $ \mu_{\mathbb S^{d-1}}$ via $U$, i.e., for any Borel set $B \subset B_{\mathbb R^k}$ we define
\begin{equation}
\label{pushforward}
\mu_k (B) := U_{\#} \mu_{\mathbb S^{d-1}}(B) := \mu_{\mathbb S^{d-1}} (U^{\leftarrow} (B)).
\end{equation}
Since $\mu_{\mathbb S^{d-1}}$ is rotation invariant, $\mu_k$ does not depend on the
particular matrix $U$, and is itself a rotation invariant measure on $B_{\mathbb R^k}$. Hence for any summable
function $h:B_{\mathbb R^k} \to \mathbb R$, for any $k\times k$ orthogonal matrix $\mathcal O$ such that
$\mathcal O \mathcal O^T = I_k = \mathcal O^T \mathcal O$, and for any $k\times d$ matrix $U$ such that $U U^T = I_k$,
we have the identities
\begin{equation}
\label{identity}
\int_{B_{\mathbb R^k}} h( \mathcal O y) d\mu_k(y) = \int_{B_{\mathbb R^k}} h(y) d\mu_k(y) =
\int_{\mathbb S^{d-1}} h(U x) d\mu_{\mathbb S^{d-1}}(x).
\end{equation}

The following result is well known. We refer to \cite[Section 1.4.4]{ru80} for the
case of ${\mathbb C}^n$. The proof given there works literally also in the real case.

\begin{theorem}\label{thmpush} Let $1\le k < d$ be natural numbers. Then the measure $\mu_{k}$ defined in \eqref{pushforward}
is given by
$$
 d\mu_{k}(y)=\frac{\Gamma(d/2)}{\pi^{k/2}\Gamma((d-k)/2)}(1-\|y\|_{\ell_2^k}^2)^{\frac{d-2-k}{2}}dy.
$$
\end{theorem}

Notice that as $d \to \infty$, and for fixed $k$, the measure $\mu_k$ becomes more and more concentrated
around $0$, in the sense that, for $\varepsilon>0$ fixed
$$
\mu_k(B_{\mathbb R^k}(\varepsilon)) \to 1, \mbox{ for } d \to \infty,
$$
very rapidly (typically exponentially).
By using the explicit form of the measure $\mu_k$ we can compute
\begin{eqnarray*}
\mu_k(B_{\mathbb R^k}(\varepsilon)) &=& 1-\frac{\Gamma(d/2)}{\pi^{k/2}\Gamma((d-k)/2)}
\int_{B_{\mathbb R^k}\setminus B_{\mathbb {R}^k}(\varepsilon)}(1-\|y\|_{\ell_2^k}^2)^{\frac{d-2-k}{2}}dy\\
&=&1-\frac{2\Gamma(d/2)}{\Gamma(k/2)\Gamma((d-k)/2)}\int_\varepsilon^1(1-r^2)^{\frac{d-2-k}{2}}r^{k-1}dr\\
&\ge&1-\frac{2\Gamma(d/2)}{\Gamma(k/2)\Gamma((d-k)/2)}e^{-\frac{d-2-k}{2}\varepsilon^2}.
\end{eqnarray*}
By Stirling's approximation $ \frac{2\Gamma(d/2)}{\Gamma(k/2)\Gamma((d-k)/2)}
\approx \sqrt{\frac{d^{d-1}}{\pi k^{k-1}(d-k)^{d-k-1}}}$, thus for $k$ and $\varepsilon$ constant
$$
\mu_k(B_{\mathbb R^k}(\varepsilon)) \to 1
$$
exponentially fast as $d \to \infty$. For $k=1$,
this phenomenon can be summarized informally by saying that the surface
measure of the unit sphere in high dimension is concentrated around the equator \cite{le01}.
Hence in case $d \gg k$ we may want to take into account possible rescaling, i.e., working with spheres of
larger radii, in order to eventually consider properties of $g$ (actually the matrix $H_g$) on larger subsets
of $\mathbb R^k$, see also Remark~\ref{remark2}. Without loss of generality, by keeping in mind this possible rescaling,
we can therefore assume to work with the unit sphere.\\


For $k=1$, we observe, that $\alpha$ as in \eqref{alpha} is determined by the interplay between the variation properties of $g$ and the measure $\mu_1.$
As just mentioned above, the most relevant feature of $\mu_1$ is that it concentrates around zero exponentially
fast as $d \to \infty$. Hence, the asymptotic behavior of $\alpha$ exclusively depends 
on the behavior of the function $g'$ in a neighborhood of $0$.


To illustrate this phenomenon more precisely, we present the following result.
\begin{prop} \label{alphad} 
Let us fix $M \in \mathbb N$ and assume that $g:B_{\mathbb R} \to \mathbb R$ is $C^{M+2}$-differentiable in an open 
neighborhood $\mathcal U$ of $0$ and $\frac{d^\ell}{{d x}^\ell}g(0) =0$ for $\ell=1,\dots,M$. Then
\begin{eqnarray*}
\alpha (d) &=& \frac{\Gamma(d/2)}{\pi^{1/2} \Gamma((d-1)/2)} \int_{-1}^1 | g'(y) |^2 (1- y^2)^{\frac{d-3}{2}} d y 
= \mathcal O(d^{-M}), \mbox{ for } d \to \infty.
\end{eqnarray*}
\end{prop}
\begin{proof}
First of all, we compute the $\ell^{th}$ moment of the measure $\frac{\Gamma(d/2)}{\pi^{1/2} \Gamma((d-1)/2)} (1- y^2)^{\frac{d-3}{2}} \mathcal L^1$:
\begin{equation}
\label{moment}
\frac{\Gamma(d/2)}{\pi^{1/2} \Gamma((d-1)/2)} \int_{-1}^1 y^\ell (1- y^2)^{\frac{d-3}{2}} d y = 
\frac{ [1 + (-1)^\ell] \Gamma(d/2) \Gamma((1+\ell)/2)}{2\sqrt \pi \Gamma((d+\ell)/2)}.
\end{equation}
Notice that all the odd moments vanish.
By Taylor expansion of $g'$ around $0$ and by taking into account that 
$\frac{d^\ell}{{d x}^\ell}g(0) =0$ for $\ell=1,\dots,M$, we obtain
$$
g'(y) = \sum_{\ell=1}^{M+1} \frac{1}{(\ell-1)!} \frac{d^\ell}{{d x}^\ell}g(0)y^{\ell-1} + \mathcal O(y^{M+1}) = 
\frac{1}{M!} \frac{d^{M+1}}{{d x}^{M+1}}g(0)y^{M} + \mathcal O(y^{M+1}).
$$
Hence,
$$
|g'(y)|^2 = \left (\frac{1}{M!} \frac{d^{M+1}}{{d x}^{M+1}}g(0) \right)^2 y^{2 M} + \mathcal O(y^{2 M +1}), 
$$
and
\begin{eqnarray*}
\alpha (d) &=& \frac{\Gamma(d/2)}{\pi^{1/2} \Gamma((d-1)/2)} \int_{-1}^1 | g'(y) |^2 (1- y^2)^{\frac{d-3}{2}} d y \\
&=& \frac{\Gamma(d/2)}{\pi^{1/2} \Gamma((d-1)/2)} \left ( \int_{\mathcal U} | g'(y) |^2 (1- y^2)^{\frac{d-3}{2}} d y + \int_{B_{\mathbb R} \setminus \mathcal U} | g'(y) |^2 (1- y^2)^{\frac{d-3}{2}} d y \right)\\
&=& \frac{\Gamma(d/2)}{\pi^{1/2} \Gamma((d-1)/2)} \left ( \left (\frac{1}{M!} \frac{d^{M+1}}{{d x}^{M+1}}g(0) \right)^2 \int_{\mathcal U}  y^{2 M} (1- y^2)^{\frac{d-3}{2}} d y \right . \\
&& \phantom{XXXXXXXXX} \left . + \int_{\mathcal U}  \mathcal O(y^{2 M +2}) (1- y^2)^{\frac{d-3}{2}} d y + \int_{B_{\mathbb R} \setminus \mathcal U} | g'(y) |^2 (1- y^2)^{\frac{d-3}{2}} d y \right).
\end{eqnarray*}
Notice that we consider the $(2M +2)^{th}$ moment in the expression above because the previous one is odd and therefore vanishes.
Now, the term $\int_{B_{\mathbb R} \setminus \mathcal U} | g'(y) |^2 (1- y^2)^{\frac{d-3}{2}} d y$ goes to zero exponentially fast for $d \to 0$.
By using \eqref{moment} we immediately obtain
\begin{eqnarray*}
\alpha (d) &=& \frac{\Gamma(d/2)}{\pi^{1/2} \Gamma((d-1)/2)} \int_{-1}^1 | g'(y) |^2 (1- y^2)^{\frac{d-3}{2}} d y \\
&=& \mathcal O\left (\frac{\Gamma(d/2)\Gamma((1+2M)/2)}{\Gamma((d+2M)/2)} \right), \quad d \to \infty.
\end{eqnarray*}
By Stirling's approximation, for which $\Gamma(z) = \sqrt{\frac{2 \pi}{z}} \left (\frac{z}{e} \right)^z + \mathcal O(1 +1/z)$, for $z \to \infty$, we obtain
$$
\frac{\Gamma(d/2)\Gamma((1+2M)/2)}{\Gamma((d+2M)/2)} \approx d^{(d-1)/2} (1 + 2 M)^M ( d+ 2 M)^{-(\frac{d+1}{2} +M)}, \quad d \to \infty.
$$
This eventually yields
\begin{align*}
\alpha (d) = \frac{\Gamma(d/2)}{\pi^{1/2} \Gamma((d-1)/2)} \int_{-1}^1 | g'(y) |^2 (1- y^2)^{\frac{d-3}{2}} d y = \mathcal O\left (d^{-M} \right), \quad d \to \infty.
\end{align*}
\end{proof}
The number $m_{\mathcal X} \times (m_{\Phi}+1)$  of points we need in order to achieve a prescribed accuracy  in the error 
estimate \eqref{error:kis1} of Theorem \ref{thm:kis1} depends on $\alpha$.
Proposition \ref{alphad} ensures that, if $g'(y)$ does not vanish for $y \to 0$ super-polynomially, then the dependence of $\alpha$ 
(therefore of the error estimate and the number   $m_{\mathcal X} \times (m_{\Phi}+1)$ of points) on $d$  is at most polynomial. 
According to this observation we distinguish three classes of ridge functions:
\begin{itemize}
\item[(1)] For $0<q\le 1$, $C_1>1$ and $C_2\geq \alpha_0>0$, we define 
\begin{eqnarray*}
\mathcal F_d^1 &:=& \mathcal F_d^1(\alpha_0, q, C_1, C_2):=\{ f:B_{\mathbb R^d}  \to \mathbb R:\\
&&\exists a\in\mathbb R^d, \|a\|_{\ell_2^d}=1, \|a\|_{\ell_q^d}\le C_1\quad\text{and}\\
&&\exists g \in C^2(B_{\mathbb R}), \ |g'(0)|\geq \alpha_0>0: f(x) = g(a \cdot x)\,\}.
\end{eqnarray*}
\item[(2)] For a neighborhood $\mathcal U$ of 0, $0<q\le 1$, $C_1>1$, $C_2\geq \alpha_0>0$ and $N\geq 2$, we define 
\begin{eqnarray*}
\mathcal F_d^2 &:=& \mathcal F_d^2(\mathcal U,\alpha_0, q, C_1, C_2, N):=\{ f:B_{\mathbb R^d}  \to \mathbb R:\\
&&\exists a\in\mathbb R^d, \|a\|_{\ell_2^d}=1, \|a\|_{\ell_q^d}\le C_1\quad\text{and}
\quad\exists g \in C^2(B_{\mathbb R}) \cap C^N(\mathcal U)\\
&&\exists 0\leq M \leq N-1, \ |g^{(M)}(0)|\geq \alpha_0>0: f(x) = g(a \cdot x)\,\}.
\end{eqnarray*}
\item[(3)] For a neighborhood $\mathcal U$ of 0, $0<q\le 1$, $C_1>1$ and $C_2\geq \alpha_0>0$, we define 
\begin{eqnarray*}
\mathcal F_d^3 &:=& \mathcal F_d^3(\mathcal U,\alpha_0, q, C_1, C_2):=\{ f:B_{\mathbb R^d}  \to \mathbb R:\\
&&\exists a\in\mathbb R^d, \|a\|_{\ell_2^d}=1, \|a\|_{\ell_q^d}\le C_1\quad\text{and}
\quad\exists g \in C^2(B_{\mathbb R}) \cap C^\infty(\mathcal U)\\
&&|g^{(M)}(0)|=0 \quad \text{for all} \quad M\in \mathbb N: f(x) = g(a \cdot x)\,\}.
\end{eqnarray*}
\end{itemize} 
Theorem \ref{thm:kis1} and Proposition \ref{alphad} immediately imply the following tractability result for
these function classes.
\begin{corollary}
The problem of learning functions $f$ in the classes $\mathcal F_d^1$ and $\mathcal F_d^2$ from point evaluations is 
\emph{strongly polynomially tractable} and \emph{polynomially tractable} respectively.
\end{corollary}
On the one hand, let us notice that if in the class  $\mathcal F_d^3$ we remove the condition  $\|a\|_{\ell_q^d}\le C_1$, then the discussion on the functions
described in Remark \ref{conditionref} shows that the problem actually becomes {\it intractable}. On the other hand, we conjecture that the restriction imposed by a condition such
as  $\|a\|_{\ell_q^d}\le C_1$ should instead give to the problem some sort of tractability.
Unfortunately, our learning method and approximation estimates in Theorem \ref{thm:kis1} do not provide any information about the tractability of the problem for functions in the class $\mathcal F_d^3$.

\section{The General Case $k\geq 1$}\label{sec:klarger1}

In this section we generalize our approach to the case $k \geq 1$, i.e., we consider {\it $k$-ridge functions}
\begin{equation}
\label{kridge}
f(x) = g (A x).
\end{equation}
Obviously, the sum of $k$ ridge functions (as appearing for example in \eqref{kridge'}) is a $k$-ridge function and the same
holds true also for the product.

We will proceed as in the one-dimensional case, giving first the basic ideas, which motivate the recovery algorithm and then stating and proving our main theorem. Remember that we assume, that $A$ is a $k \times d$ matrix such that $A A^T = I_k$, 
and $g:B_{\mathbb R^k}(1+\bar\epsilon)\to \mathbb R$ is a $C^2$ function.

\subsection{The Algorithm}\label{sec:algklarger1}

As before we consider a version of Taylor's theorem giving access to the matrix $A$. For $\xi \in B_{\mathbb R^d}$, $\varphi \in B_{\mathbb R^d}(r)$, $\epsilon,r \in \mathbb R_+$, 
with $r \epsilon \leq \bar \epsilon$, we have the identity
\begin{eqnarray}
[\nabla g(A \xi)^T A ] \varphi &=&  \frac{f(\xi + \epsilon \varphi) - f(\xi)}{\epsilon} - \frac{\epsilon}{2} 
[\varphi^T \nabla^2 f(\zeta) \varphi]\label{taylor2},
\end{eqnarray}
for a suitable $\zeta(\xi,¸\varphi) \in B_{\mathbb R^d}(1+\bar \epsilon)$ and thanks to \eqref{cond2} the term $[\varphi^T \nabla^2 f(\zeta) \varphi]$ is again uniformly bounded as soon as $\varphi$ is bounded.\\
As in the one-dimensional case we now consider~\eqref{taylor2} for the $m_\Phi$ directions in the set $\Phi$ and at the $m_{\mathcal X}$ sampling points in the set $\mathcal X$, where $\mathcal X, \Phi$ are 
defined as in \eqref{Xpoints} and \eqref{Phipoints} respectively.
Again we collect the directional derivatives $\nabla g(A \xi_j)^T A$, $j=1,\dots, m_{\mathcal X}$ as columns in the $d \times m_{\mathcal X}$ matrix $X$, i.e., 
\begin{equation}
X=(A^T \nabla g(A \xi_1),\ldots, A^T \nabla g(A \xi_{m_\mathcal{X}})),
\end{equation}
and using the matrices $Y$ and $\mathcal E$ as defined in \eqref{yij} and \eqref{eij}, we can write the following factorization
\begin{equation}
\label{factor}
\Phi X = Y - \mathcal E.
\end{equation}

Similarly to the one-dimensional case we find that the matrix $X$ has a special structure, which we will exploit for the algorithm, i.e.,
$X= A^T \mathcal G^T$, where $\mathcal G = ( \nabla g(A \xi_1)^T| \dots | \nabla g(A \xi_{m_\mathcal{X}})^T)^T$. 
The columns of $X$ are now no longer scaled copies of one compressible vector but they are linear combinations
of $k$ compressible vectors, i.e., the rows of the matrix $A$.
Thus compressed sensing theory again tells us that we can stably recover the columns of $X$ from the columns of $Y$ via 
$\ell_1$-minimization and in consequence get a good approximation $\hat X$ to $X$.\\
Furthermore, since $A$ has rank $k$, as long as $\mathcal G^T$ has full rank, also $X$ will have rank $k$ and moreover the column span of the right singular vectors of $X^T = USV^T$ will coincide with the row span of $A$, i.e., $A^TA=VV^T$. Moreover, $V^T$ gives us an alternative representation of $f$ as follows, i.e.,
$$
f(x)=g(Ax)=g(AA^TAx)=g(AVV^Tx)=: \tilde g(V^Tx),
$$
where $\tilde g(y):=g(AVy)=f(Vy)$. If $\hat X$ is a good approximation to $X$, then we can expect that the first $k$ right singular vectors of $\hat X$ have almost the same span as those of $X$ and thus of $A$, which inspires the following algorithm.\\

\fbox{
\begin{minipage}{14cm}
\noindent{\bf Algorithm 2}:
{\it \begin{itemize}
\item Given $m_\Phi, m_{\mathcal X}$, draw at random the sets $\Phi$ and $\mathcal X$ as in \eqref{Xpoints} and \eqref{Phipoints}, and construct $Y$ according to~\eqref{yij}.
\item Set $\hat x_j = \Delta (y_j) := \argmin_{y_j=\Phi z } \| z\|_{\ell_1^d}$, for $j=1,\dots,m_{\mathcal X}$, and $\hat X = (\hat x_1,\ldots, \hat x_{m_\mathcal X})$.
\item Compute the singular value decomposition of 
\begin{equation}\label{SVD}
\hat X^T = \left (\begin{array}{lll}
\hat U_{1}&\hat U_{2}\end{array}\right )
\left (\begin{array}{ll}\hat \Sigma_{1}& 0\\
0& \hat \Sigma_{2 }\\\end{array}\right )
\left (\begin{array}{l}\hat V_{1 }^T\\\hat V_{2}^T\end{array}
\right ),
\end{equation}
where $\hat \Sigma_{1}$ contains the $k$ largest singular values.
\item Set $\hat A = V_{1 }^T$.
\item Define $\hat g (y):= f(\hat A^T y)$ and $\hat f (x) := \hat g (\hat A x)$.
\end{itemize}}
\end{minipage}
}
\\

The quality of the final approximation of $f$ by means of $\hat f$ depends on two kinds of accuracies: 
\begin{itemize}
\item[1. ] The error between $\hat X$ and $X$, which can be controlled through the number of compressed sensing measurements $m_\Phi$;
\item[2. ] The stability of the span of $V^T$, simply characterized by how well the singular values of $X$ or equivalently $\mathcal G$ are separated from 0, which is related to the number of random samples $m_\mathcal{X}$. 
\end{itemize}
To be precise, in the next section we will prove the following approximation result.

\begin{theorem} \label{thm:klarger1}
Let $\log d\le m_\Phi\le [\log 6]^2 d$. Then there is a constant $c'_1$ such that
using $m_\mathcal{X} \cdot (m_\Phi+1)$ function evaluations of $f$, Algorithm 2 defines a function 
$\hat f:B_{\mathbb R^d}(1+\bar \epsilon) \to \mathbb R$ that, with probability 
\begin{equation}
1-\left (e^{-c'_1 m_\Phi}+e^{-\sqrt{m_\Phi d}}+ k e^{\frac{- m_{\mathcal X}\alpha s^2  }{2 k C_2^2}} \right),\label{prob:klarger1}
\end{equation}
will satisfy
\begin{equation}
\|f -\hat f\|_\infty \leq 2C_2\sqrt{k}(1+\bar \epsilon)\frac{\nu_2}{\sqrt{\alpha (1-s)} - \nu_2},\label{error:klarger1}
\end{equation}
where
$$
\nu_2=C\left(k^{1/q}\left[\frac{m_\Phi}{\log(d/m_\Phi)}\right]^{1/2-1/q}+\frac{\epsilon k^2}{\sqrt{m_\Phi}}\right),
$$
and $C$ depends only on $C_1$ and $C_2$ (cf. \eqref{cond1} and \eqref{cond2}).
\end{theorem}


\subsection{The Analysis}
We will first show that $\hat X$ is a good approximation to $X$ by applying Theorem~\ref{thm:cs} columnwise.
This leads to the following corollary.

\begin{corollary}\label{cor2}
Let $\log d\le m_\Phi < [\log 6]^2 d$. Then with probability 
$$
1-(e^{-c'_1 m_\Phi}+e^{-\sqrt{m_\Phi d}})
$$
the matrix $\hat X$ as calculated in Algorithm 2 satisfies
\begin{equation}\label{approx22}
\| X- \hat X\|_{F} \leq C\sqrt{m_{\mathcal X}}\left(k^{1/q}\left[\frac{m_\Phi}{\log(d/m_\Phi)}\right]^{1/2-1/q}+\frac{\epsilon k^2}{\sqrt{m_\Phi}}\right),
\end{equation}
where $C$ depends only on $C_1$ and $C_2$ (cf. \eqref{cond1} and \eqref{cond2}).
\end{corollary}


\begin{proof}The proof works essentially like that of Corollary \ref{cor1}.
We decompose 
$$
\| X - \hat X\|_{F}^2 = \sum_{j=1}^{m_{\mathcal X}} \| x_j - \hat x_j\|_{\ell_2^d}^2.
$$
The best $K$-term approximation of $x_j$ may be estimated using
\begin{align*}
\|x_j\|_{\ell_q^d}=\|A^T\nabla g(A \xi_j)\|_{\ell_q^d}\le C_2 \left(\sum_{v=1}^d\left(\sum_{u=1}^k|a_{uv}|\right)^q\right)^{1/q}
\le C_1\,C_2\,k^{1/q},
\end{align*}
which leads to
$$
K^{-1/2}\sigma_K(x_j)_{\ell_1^d}\le \|x_j\|_{\ell_q^d}\,K^{1/2-1/q} \le C_1\,C_2\,k^{1/q} K^{1/2-1/q}.
$$
The norms of $\varepsilon_j$ may be estimated similarly to the proof of Corollary \ref{cor1} as
\begin{equation*}
\|\varepsilon_j\|_{\ell_2^{m_\Phi}}\le \frac{C_1^2\,C_2\,k^2\epsilon}{2\sqrt{m_\Phi}}\quad\text{and}\quad
\|\varepsilon_j\|_{\ell_\infty^{m_\Phi}}\le \frac{C_1^2\,C_2\,k^2\epsilon}{2m_\Phi}.
\end{equation*}
Putting all these estimates (with the choice $K\approx m_\Phi/\log (d/m_\Phi)$) into Theorem \ref{thm:cs} we obtain the result.
\end{proof}

\begin{remark}
The construction $\hat x_j = \Delta (y_j) := \argmin_{y_j=\Phi z } \| z\|_{\ell_1^d}$, for $j=1,\dots,m_{\mathcal X}$, and $\hat X = (\hat x_1,\ldots, \hat x_{m_\mathcal X})$ and Corollary \ref{cor2} are not the unique possible approach to approximate $X$.
As we are expecting $X$ to be a $k$-rank matrix for $k \ll \min\{d,m_{\mathcal X}\}$, one might want to consider also \emph{nuclear norm minimization}, i.e., the minimization of the $\ell_1$-norm of singular values, as a possible way of accessing $X$ from $m_\Phi$ random measurements, as in the work \cite{fa02-2,oymofaha11,fapareXX}. However, presently no estimates of the type \eqref{woj1} are available in this context, hence we postpone an analysis based on these methods fully tailored to matrices to further research.
\end{remark}

Next we need the equivalent of Lemma~\ref{lem1} to relate the error between the subspaces defined by the largest right singular values of $\hat X$ and $X$ respectively to the error $\| X- \hat X\|_{F}$. We will develop the necessary tools in the following subsection.

\subsubsection{Stability of the singular value decomposition}

Given two matrices $B$ and $\hat B$ with corresponding singular value decompositions
$$
B = \left (
\begin{array}{lll}
U_1&U_2
\end{array}
\right )
\left (
\begin{array}{ll}
\Sigma_1& 0\\
0& \Sigma_2\\
\end{array}
\right )
\left (
\begin{array}{l}
V_1^T\\
V_2^T
\end{array}
\right )
$$ 
and
$$
\hat B = \left (
\begin{array}{lll}
\hat  U_1& \hat U_2 
\end{array}
\right )
\left (
\begin{array}{ll}
\hat \Sigma_1& 0\\
0& \hat \Sigma_2\\
\end{array}
\right )
\left (
\begin{array}{l}
\hat V_1^T\\
\hat  V_2^T
\end{array}
\right ),
$$
where it is understood that two corresponding submatrices, e.g., $U_1,\hat U_1$, have the same size,
we would like to bound the difference between $V_1$ and $\hat V_1$ by the error $\| B - \hat B\|_F$. As a consequence of Wedin's perturbation bound \cite{we72}, see also \cite[Section 7]{st90},
we have the following useful result.
\begin{theorem}[Stability of subspaces - Wedin's bound]
\label{wedin}
If there is an $\bar \alpha >0$ such that
\begin{equation}
\label{separa1}
\min_{\ell,\hat \ell} | \sigma_{\hat \ell}(\hat \Sigma_1) - \sigma_{ \ell}( \Sigma_2) | \geq \bar \alpha,
\end{equation}
and
\begin{equation}
\label{separa2}
\min_{\hat \ell} | \sigma_{\hat \ell}(\hat \Sigma_1) | \geq \bar \alpha,
\end{equation}
then
\begin{equation}
\label{projest}
\| V_1 V_1^T - \hat V_1 \hat V_1^T\|_F \leq \frac{2}{\bar \alpha} \| B - \hat B\|_F.
\end{equation}
\end{theorem}
The conditions \eqref{separa1} and \eqref{separa2} are separation conditions. The first says that the singular values of $\Sigma_1$ are separated from those of $\Sigma_2$. Actually, strictly speaking the separation is between $ \Sigma_1$ and $\hat \Sigma_2$. However, if $ \| B - \hat B\|_F$ is sufficiently small compared to $\bar \alpha$, then Weyl's inequality \cite{we12}
$$
|\sigma_\ell(B) - \sigma_\ell(\hat B)| \leq \| B - \hat B\|_F,
$$
guarantees that the two separations are essentially equivalent. The second condition says that the singular values of $\Sigma_1$ or $\hat \Sigma_1$ have to be far away from $0$.\\
\noindent Applied to our situation, where $X$ has rank $k$ and thus $\Sigma_2 = 0$, we get
\begin{equation}
\| V_1 V_1^T - \hat V_1 \hat V_1^T\|_F \leq \frac{2 \sqrt{m_{\mathcal X}} \nu_2}{\sigma_k(\hat X^T)},
\end{equation}
and further since $\sigma_k(\hat X^T) \geq \sigma_k(X^T) - \|X - \hat X\|_F$, that
\begin{equation}
\| V_1 V_1^T - \hat V_1 \hat V_1^T\|_F \leq \frac{2\sqrt{m_{\mathcal X}} \nu_2}{\sigma_k( X^T)-\sqrt{m_{\mathcal X}} \nu_2}.
\end{equation}
As final ingredient we need to estimate the $k$-th singular value of $X$. The next subsection will provide us with a generalization of Hoeffding's inequality, that can be used to show that with high probability on the random draw of the sampling points $\xi_j$ the $k$-th singular value of
$X$ is separated from zero.

\subsubsection{Spectral estimates and sums of random semidefinite matrices}
The following theorem generalizes Hoeffding's inequality to sums of random semidefinite matrices and was recently proved by 
Tropp in \cite[Corollary 5.2 and Remark 5.3]{tr10}, improving over results in \cite{ahwi02}, and using techniques from \cite{ruve07} and \cite{ol10}.
\begin{theorem}[Matrix Chernoff]
\label{chernmat}
Consider $X_1, \dots, X_m$ independent random, positive-semidefinite matrices of dimension $k \times k$. Moreover suppose 
\begin{equation}
\label{bnd}
\sigma_1(X_j) \leq C,
\end{equation}
almost surely.
Compute the singular values of the sum of the expectations
\begin{equation}
\label{singbnd}
\mu_{\max} = \sigma_1 \left ( \sum_{j=1}^m \mathbb E X_j\right) \mbox{ and } \mu_{\min} = \sigma_k \left ( \sum_{j=1}^m \mathbb E X_j\right),
\end{equation}
then
\begin{equation}
\label{lowerbnd}
\mathbb P \left \{ \sigma_1\left ( \sum_{j=1}^m X_j\right) - \mu_{\max} \geq s\mu_{\max}  \right \} \leq k \left ( \frac{(1+s)}{ e}\right )^{-\frac{\mu_{\max}(1+s)}{C}},
\end{equation}
for all $s>(e-1)$, and
\begin{equation}
\label{upperbnd}
\mathbb P \left \{ \sigma_k\left ( \sum_{j=1}^m X_j\right) - \mu_{\min} \leq -s \mu_{\min} \right \} \leq k e^{-\frac{ \mu_{\min}s^2}{2C}},
\end{equation}
for all $s \in (0,1)$.
\end{theorem}

Applied to the matrix $X^T$ the above theorem leads to the following estimate of the singular values of $X^T$.

\begin{lemma}
\label{speclem}
For any $s\in (0, 1)$ we have that
\begin{equation}
\label{lowerbnd2}
\sigma_k(X^T) \geq \sqrt{m_{\mathcal X} \alpha (1-s)}
\end{equation}
with probability $1 - k e^{\frac{- m_{\mathcal X}\alpha s^2  }{2 k C_2^2}}$.
\end{lemma}
\begin{proof} The proof is based on an application of  Theorem \ref{chernmat}. 
First of all note that
$$
X^T = \mathcal G A = U_{\mathcal G} \Sigma_{\mathcal G} [V_{\mathcal G}^T A],
$$
hence $\Sigma_{X^T} = \Sigma_{\mathcal G}$. Moreover
$$
\sigma_i(\mathcal G) = \sqrt{ \sigma_i(\mathcal G^T \mathcal G)}, \quad \mbox{ for all } i=1,\dots,k.
$$
Thus, to get information about the singular values of $X^T$ it is sufficient to study that of
$$
\mathcal G^T \mathcal G = \sum_{j=1}^{ m_{\mathcal X}}\nabla g(A \xi_j) \nabla g(A \xi_j)^T.
$$
We further notice that
$$
\sigma_1(\nabla g(A \xi_j) \nabla g(A \xi_j)^T) \leq \left ( \sum_{\ell,\ell'=1}^k |\nabla g(A \xi_j)_\ell \nabla g(A \xi_j)_{\ell'}|^2 \right )^{1/2} \leq k C^2_2:=C.
$$
Hence $X_j = \nabla g(A \xi_j) \nabla g(A \xi_j)^T$ is a random positive-semidefinite matrix, that is almost surely bounded. Moreover
$$
\mathbb E X_j = \mathbb E_\xi  \nabla g(A \xi_j) \nabla g(A \xi_j)^T = \int_{\mathbb S^{d-1}} \nabla g(A x) \nabla g(A x)^T d\mu_{\mathbb S^{d-1}}(x) = H_g.
$$
Hence, remembering that the singular values of $H_g$ are equivalent to that of $H^f$, by condition \eqref{cond3} we have $\mu_{\max} = m_{\mathcal X} \sigma_1(H_g) \leq m_{\mathcal X} k C^2_2$ and $\mu_{\min} = m_{\mathcal X}  \sigma_k(H_g) \geq m_{\mathcal X} \alpha>0$.
In particular
$$
m_{\mathcal X} k^2 C_2 \geq \mu_{\max} \geq \mu_{\min} \geq  m_{\mathcal X} \alpha >0.
$$
By an application of Theorem \ref{chernmat} we conclude that
$$
\sigma_k(X^T) = \sigma_k(\mathcal G) = \sqrt{\sigma_k \left (\sum_{j=1}^{ m_{\mathcal X}}\nabla g(A \xi_j) \nabla g(A \xi_j)^T \right)}\geq \sqrt{\mu_{\min}(1-s)} \geq \sqrt{m_{\mathcal X} \alpha (1-s)},
$$
with probability
$$
1 - k e^{-\frac{\mu_{\min}s^2  }{2 k C^2_2}} \geq 1 - k e^{\frac{- m_{\mathcal X}\alpha s^2  }{2 k C_2^2}}, 
$$
for all $s \in (0,1)$. 
\end{proof}

\noindent Finally we have collected all the results necessary to prove Theorem~\ref{thm:klarger1}.\\

\noindent{\bf Proof of Theorem~\ref{thm:klarger1}:}
\begin{proof} 
Combining Corollary~\ref{cor2}, Theorem~\ref{wedin}, and Lemma~\ref{speclem} shows that with probability at least
$$
1-\left (e^{-c'_1 m_\Phi}+e^{-\sqrt{m_\Phi d}}+  k e^{\frac{- m_{\mathcal X}\alpha s^2  }{2 k C_2^2}} \right ),
$$
for the first $k$ right singular vectors of $\hat X$ and $X$ we have
\begin{equation}
\| V_1 V_1^T - \hat V_1 \hat V_1^T\|_F \leq \frac{2\nu_2}{\sqrt{\alpha (1-s)}-\nu_2}. \notag
\end{equation}
Recalling from the proof of Lemma~\ref{speclem} that the (first $k$) right singular vectors $V_1^T$ of 
$X^T$ have the form $V_1^T = V_{\mathcal G}^TA$ then shows that $\hat A$ 
as defined in Algorithm~2 satisfies
\begin{equation} 
\| A^T A - \hat A^T \hat A\|_F = \| A^T V_{\mathcal G} V_{\mathcal G}^TA -\hat V_1 \hat V_1^T\|_F=\| V_1 V_1^T - \hat V_1 \hat V_1^T\|_F   \leq \frac{2 \nu_2}{\sqrt{\alpha (1-s)} -\nu_2}, \notag
\end{equation}
Using this estimate we can prove that $\hat f$ as defined in Algorithm 2 is a good approximation to $f$.
Since $A$ is row-orthogonal we have $A=A A^T A$ and therefore
\begin{align*}
|f(x)-\hat f(x)|&= | g (A x) - \hat g (\hat A x) | \\
&=  | g (A x) - g (A \hat A^T\hat A x) |\\
&\leq C_2 \sqrt{k}\|A x - A \hat A^T\hat A x\|_{\ell_2^k}\\
&=C_2 \sqrt{k}\|A (A^T A - \hat A^T\hat A) x\|_{\ell_2^k}\\
&\leq C_2 \sqrt{k}\|(A^T A - \hat A^T\hat A)\|_F \|x\|_{\ell_2^d}\\
&\leq 2C_2\sqrt{k}(1+\bar \epsilon)\frac{\nu_2}{\sqrt{\alpha (1-s)} - \nu_2}.
\end{align*}
\end{proof}

\begin{remark}
\noindent (i) Note that Theorem~\ref{thm:klarger1} is again an a priori estimate of the success probability and approximation error of Algorithm 2.
Once Algorithm 2 has been run we have the following a posteriori estimate. 
With probability at least $1-(e^{-c'_1 m_\Phi}+e^{-\sqrt{m_\Phi d}})$ we have that 
$$
\| f - \hat f \|_\infty  \leq 2C_2\sqrt{k m_{\mathcal X}}(1+\bar \epsilon) \frac{\nu_2} { \sigma_k(\hat X^T)}.
$$\\
\noindent (ii) We further observe that Theorem~\ref{thm:klarger1} does not straightforwardly reduce to Theorem~\ref{thm:kis1}
for $k=1$, because in the one-dimensional case we used the simpler maximum strategy as in \eqref{maxchoice} instead of the
singular value decomposition \eqref{SVD}.
\end{remark}


\subsection{Discussion on tractability}
Recall, that the push-forward measure $\mu_k = \frac{\Gamma(d/2)}{\pi^{k/2} \Gamma((d-k)/2)}
(1- \|y\|_{\ell_2^k}^2)^{\frac{d-2-k}{2}} \mathcal L^k$ of $\mu_{\mathbb S^{d-1}}$ on the unit ball $B_{\mathbb R^k}$
was determined in Theorem \ref{thmpush} as the measure, for which
\begin{eqnarray*}
H_g  &=& \int_{\mathbb S^{d-1}} \nabla g(A  x)  \nabla g(A  x)^T d \mu_{\mathbb S^{d-1}}(x)\\
&=&\frac{\Gamma(d/2)}{\pi^{k/2} \Gamma((d-k)/2)} \int_{B_{\mathbb R^k}}  \nabla g(y)  \nabla g(y)^T  
(1- \|y\|_{\ell_2^k}^2)^{\frac{d-2-k}{2}} d y.
\end{eqnarray*}
As an instructive example, let us apply this formula to the case when $g$ is a radial function, i.e.,
$$
g(y) = g_0(\|y\|_{\ell_2^k}),
$$
for a function $g_0:[0,1] \to \mathbb R$ sufficiently smooth, and $g_0'(0) =0$. 

A direct calculation shows, that $\nabla g(y)=\frac{g_0'(r)}{r}\cdot y$, where $r=\|y\|_{\ell_2^k}$,
and 
$$
\nabla g(y)  \nabla g(y)^T  = \frac{g'_0(r)^2}{r^2}y y^T.
$$

Hence,
$$
(H_g)_{ij}=
\frac{\Gamma(d/2)}{\pi^{k/2}\Gamma((d-k)/2)}\int_{B_{\mathbb R^k}}
\frac{g'_0(\|y\|_{\ell_2^k})^2}{\|y\|_{\ell_2^k}^2}y_iy_j(1-\|y\|_{\ell_2^k}^2)^{\frac{d-2-k}{2}}dy.
$$
If $i\not=j$, the integral vanishes due to the symmetry of $B_{\mathbb R^k}$. If $i=j$, we get again by symmetry
\begin{eqnarray*}
(H_g)_{ii}&=&
\frac{\Gamma(d/2)}{\pi^{k/2}\Gamma((d-k)/2)}\int_{B_{\mathbb R^k}}\frac{g'_0(\|y\|_{\ell_2^k})^2}{\|y\|_{\ell_2^k}^2}y^2_i(1-\|y\|_{\ell_2^k}^2)^{\frac{d-2-k}{2}}dy\\
&=& \frac{\Gamma(d/2)}{k\pi^{k/2}\Gamma((d-k)/2)}\int_{B_{\mathbb R^k}}g'_0(\|y\|_{\ell_2^k})^2(1-\|y\|_{\ell_2^k}^2)^{\frac{d-2-k}{2}}dy\\
&=& \frac{2\Gamma(d/2)}{k\Gamma((d-k)/2)\Gamma(k/2)}\int_{0}^1g'_0(r)^2(1-r^2)^{\frac{d-2-k}{2}}r^{k-1}dr=:\alpha(k,d).
\end{eqnarray*}
Hence, $H_g=\alpha(k,d) I_k$.
Similarly to Proposition \ref{alphad}, we can expand $g_0'$ into a Taylor series
$$
g_0'(r) = \sum_{\ell=2}^{N-1} \frac{g_0^{(\ell)}(0)}{(\ell-1)!} r^{\ell-1} + \mathcal O(r^N).
$$
If we assume that $g_0^{(\ell)}(0) =0$, for all $\ell=1,\dots,M$, but $g_0^{(M+1)}(0)\neq 0$, then we obtain
$$
g_0'(r)^2 = \left ( \frac{g_0^{(M+1)}(0)}{M!} \right )^2 r^{2M} + \mathcal O(r^{2M +1}),
$$
and, by Stirling's approximation,
\begin{eqnarray*}
\alpha(k,d) &=&  \mathcal O \left ( \frac{\Gamma(d/2)}{\Gamma((d-k)/2)}  \int_0^1 r^{2M+k-1} (1- r^2)^{\frac{d-k-2}{2}} d r \right )\\
&=&  \mathcal O \left ( \frac{\Gamma(d/2)}{\Gamma(d/2+M)} \right )\\
&=& \mathcal O \left ( d^{-M} \right ), \quad d \to \infty.
\end{eqnarray*}
From these computations, we deduce that learning functions $f(x) = g(A x)$, where $g$ is radial (or nearly radial), 
using our method has usually polynomial complexity with respect to the dimension $d$.

\section{Extensions and Generalizations}

We assumed throughout the paper that the function $f$ is defined on the unit ball $B_{\mathbb R^d}$
of $\mathbb R^d$. To be able to approximate the derivatives of $f$ even on the boundary of $B_{\mathbb R^d}$,
we actually supposed, that $f$ is defined also on an $\bar\epsilon$ neighborhood of the unit ball.
Furthermore, we assumed that the function values may be measured exactly without any error.
The main aim of this section is to discuss the possibilities and limitations of our method. Firstly, we discuss
the numerical stability of our approach with respect to noise. Secondly, we deal
with functions defined on a convex body $\Omega\subset \mathbb R^d.$
As it is our intention here only to sketch, still rigorously, further interesting research directions, we limit our discussion to 
the case of $k=1$.

\subsection{Stability under noisy measurements}\label{sec:noise}

Let us assume that the function evaluation in \eqref{taylor} can be performed only with certain precision. We again collect
the $m_{\mathcal X}\times m_{\Phi}$ instances of \eqref{taylor} as
\begin{equation}\label{eq:noise1}
\Phi X=Y-{\mathcal E}+\frac{\mathcal W}{\epsilon},
\end{equation}
where the $(i,j)$ entry of ${\mathcal W}$ (denoted by $w_{ij}$) is the difference between the exact value
of $f(\xi_j+\epsilon\varphi_i)-f(\xi_j)$ and its value measured with noise. This leads to a compressed sensing setting
\begin{equation}\label{eq:noise2}
Y=\Phi X+{\mathcal E}-\frac{{\mathcal W}}{\epsilon}.
\end{equation}
Applying Theorem \ref{thm:cs} we obtain a substitute for Corollary \ref{cor1} with ${\mathcal E}$ replaced by
${\mathcal E}-{\mathcal W}/\epsilon.$ Therefore we would like to estimate the norm of $w_j$ (the $j$-th column of ${\mathcal W}$)
in $\ell_2^{m_\Phi}$ and $\ell_\infty^{m_\Phi}$. If we merely assume that the noise is bounded (i.e. $|w_{ij}|\le \nu$), the
best possible estimate is $\|w_j\|_{\ell_2^{m_\Phi}}\le \nu\sqrt{m_\Phi}$. We observe that the more sampling points we take
the greater is the level of noise. This effect of amplification of the noise is actually known under the name of \emph{noise folding} \cite{cael11} and, unfortunately,
corrupts the estimate \eqref{approx2-1}, see also \cite[Section 4]{codadekepi10} for a discussion in a related context.

Let us therefore sketch a different approach. We make the rather natural assumption that $w_{ij}$ is a random noise.

The analogue of Theorem \ref{thm:cs} for the recovery of $x$ from noisy measurements $y=\Phi x+\omega$, 
where $\omega=(\omega_1,\dots,\omega_m)$ are independent identically distributed (i.i.d.) Gaussian variables with mean zero and variance $\sigma^2$,
was given in the work of Cand\`es and Tao \cite{cata07}.
They proposed a certain $\ell_1$-regularization problem, whose solution (called the \emph{Dantzig selector})
satisfies
$$
\|x-\hat x\|^2_{\ell_2^d}\le C^2\cdot 2\log d\cdot \left(\sigma^2+\sum_{i=1}^d\min(x_i^2,\sigma^2)\right).
$$
Especially, if $x$ is a $k$-sparse vector, then 
$\|x-\hat x\|_{\ell_2^d}\le C\cdot \sqrt{2\log d}\cdot \sqrt{k+1}\cdot\sigma$.
We observe that this estimate scales very favorably with $d$ (only as $\sqrt{\log d}$) and, moreover,
does depend only on the sparsity of $x$, and not anymore on the number of measurements $m_\Phi$. Therefore, there is no noise folding in this case.

The equation \eqref{eq:noise2} requires a combination of Theorem \ref{thm:cs} and the result of Cand\`es and Tao.
Namely, we would like to reconstruct $x$ if $y=\Phi x+\varepsilon+\omega$ is given, where $\varepsilon$ is a deterministic error
and $\omega$ is a vector of i.i.d. Gaussian variables. Obviously, the detailed analysis of this issue goes beyond the scope of this paper.
Nevertheless, let us present some numerical evidence of the numerical stability of our approach in the presence of random noise.

We consider the function
\begin{equation}\label{eq:funct:noise}
f(x)=\max\left(\left[1-5\sqrt{(x_3-1/2)^2+(x_4-1/2)^2}\right]^3,0\right),\quad x\in\R^{1000}
\end{equation}
in dimension $d=1000$. We use a variant of Algorithm 1 based on $\ell_1$ minimization
to identify the active coordinates of $f$, cf. \cite{scvy11} for details. We suppose that function evaluations were distorted by
Gaussian error $\nu \omega$ with $\omega\approx {\mathcal N}(0,1)$ and $\nu\in\{0.1, 0.01, 0.001\}$.
We chose $\epsilon=0.1$ in the approximation \eqref{taylor}.
For each number of points $m_{\mathcal X}\in\{6 \ell, \ell=1,\dots,10\}$ ($x$-axis) and each number 
of directions $m_\Phi\in\{20 \ell,\ell=1,\dots,10\}$ ($y$-axis) we produced one hundred trials.
The success rates of recovery go from white color (no success) to black (100 successful  recoveries).\\
\hskip-.5cm{\includegraphics[width=5.3cm]{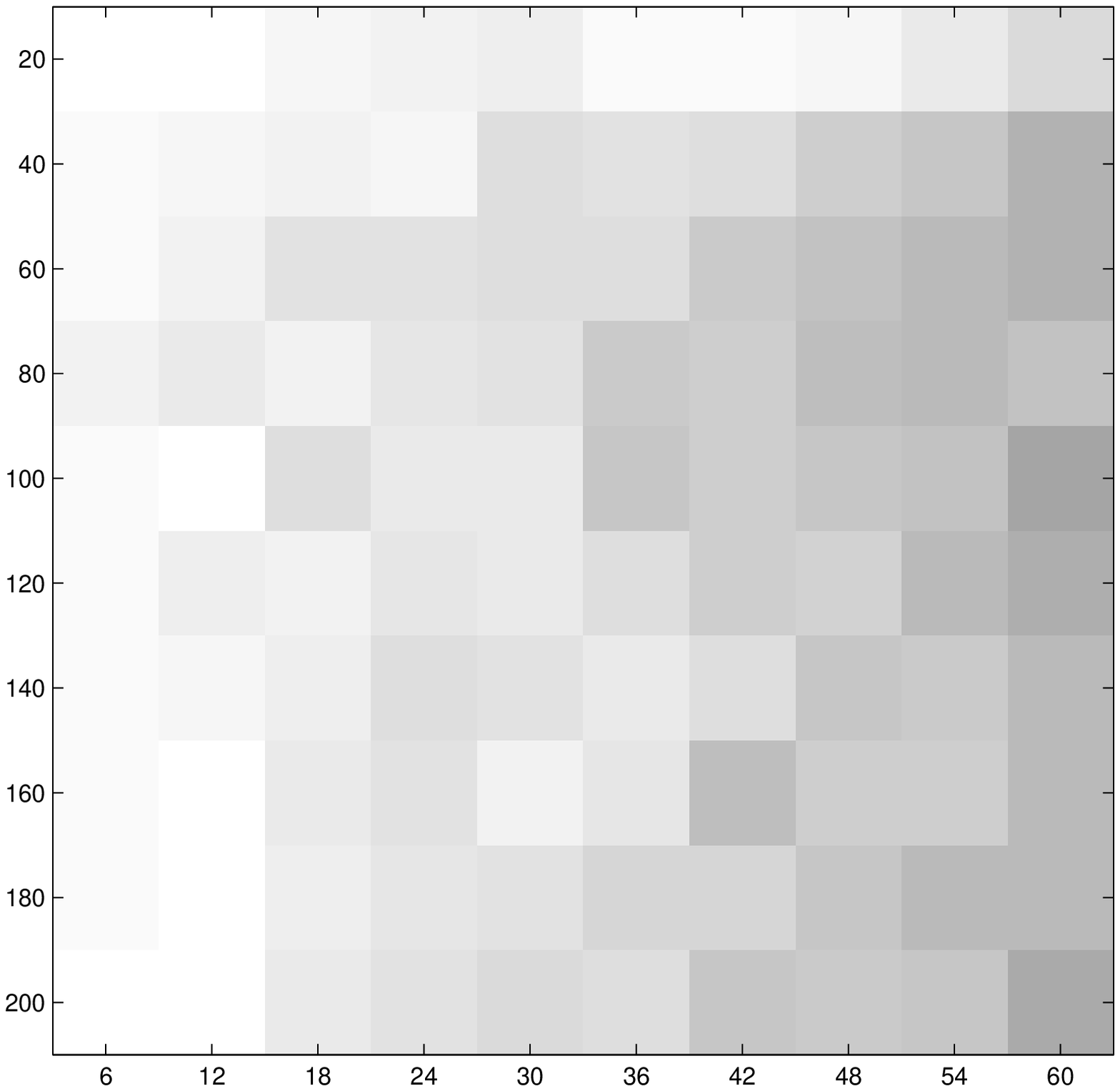}}
\hskip-.5cm{\includegraphics[width=5.3cm]{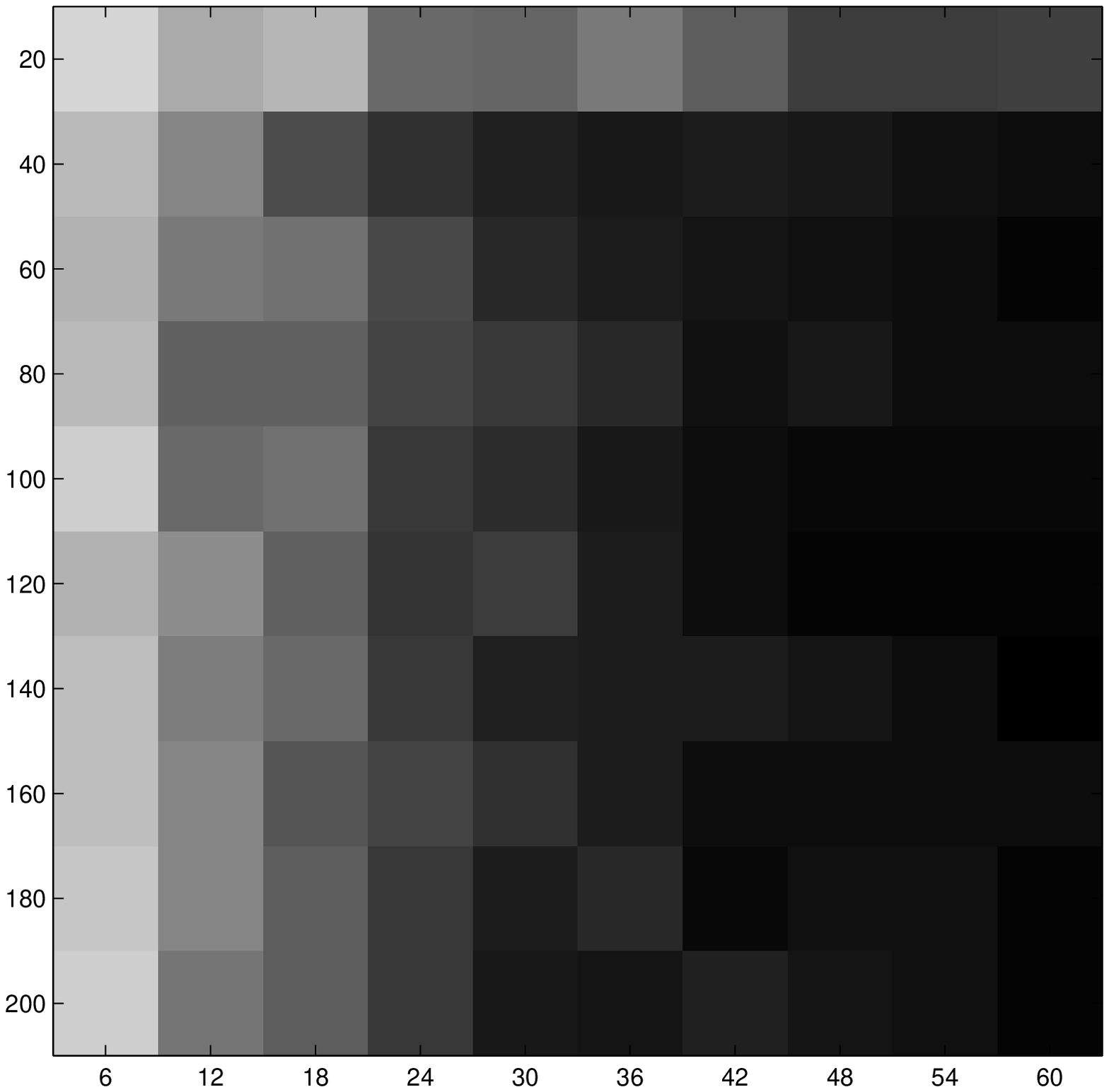}}
\hskip-.5cm{\includegraphics[width=5.8cm]{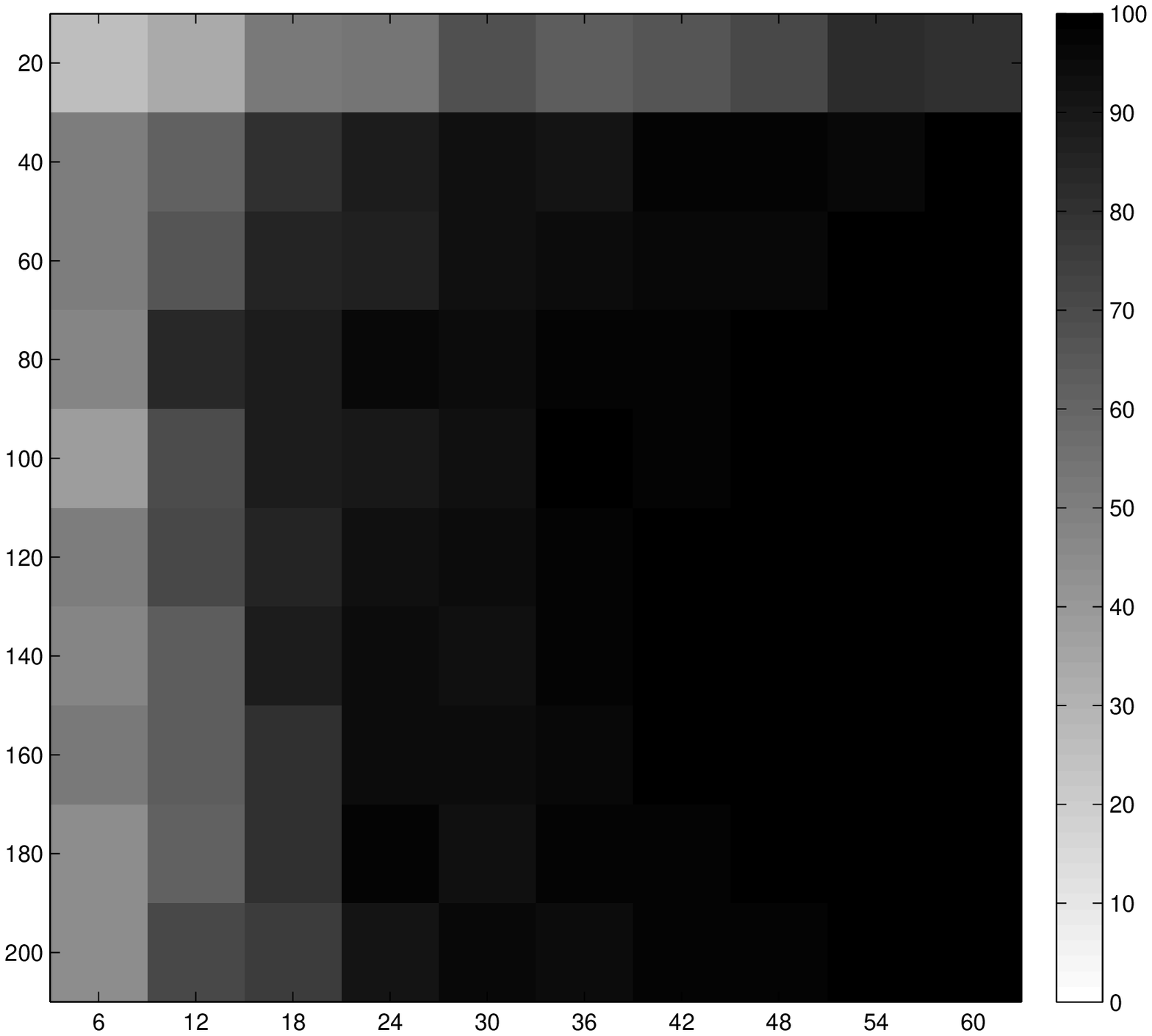}}
Figure 2: Recovery of active coordinates of $f(x)$ given by \eqref{eq:funct:noise} with
$\nu=0.1$, $\nu=0.01$ and $\nu=0.001$, from left to right respectively.
Let us mention that the success rates of recovery for noise-free setting are hardly distinguishable from the last picture
above ($\nu=0.001$).
\\

We conclude from Figure 2 that there is a smooth increase of the rate of successful recovery with decreasing noise power and a fully
stable recovery behavior. 
\subsection{Convex bodies}

A careful inspection of our method shows, that it may be generalized to arbitrary convex bodies. Let us describe the necessary modifications and give an overview of the results for the case $k=1$. 
First of all, one has to replace \eqref{eq:def:hf} by 

\begin{equation}\label{eq:def:hf2}
H^f := \int_{\Omega} \nabla f(x) \nabla f(x)^T d\mu_{\Omega}(x).
\end{equation}
Here, $\mu_\Omega$ is a probability measure on $\Omega$ and the points in $\mathcal X$ (cf. \eqref{Xpoints})
are selected at random with respect to $\mu_{\Omega}$. For $\Omega=B_{\mathbb R^d}$, we simply selected $\mu_{\Omega}=\mu_{\mathbb S^{d-1}}$ 
to be the normalized surface measure on $\mathbb S^{d-1}$. This corresponded to the fact, 
that $a\in \mathbb S^{d-1}$ was arbitrary and therefore a-priori no direction was preferred. To be able to evaluate the derivatives of 
$f$ even on the boundary of $\Omega$, we suppose, that $f$ is actually defined on an $\bar\epsilon$ neighborhood of $\Omega$, namely 
on the set $\Omega+\bar\epsilon:=\{x\in\mathbb R^d: {\rm dist}(\Omega,x)\le \bar\epsilon\}$. The function $g$ is supposed to be 
defined on the image of $\Omega+\bar\epsilon$ under the mapping $x\to a\cdot x$, i.e., on an interval. We assume again \eqref{cond2}.\\
Surprisingly enough, these are all the modifications necessary to proceed with the identification of $\hat a$ and
\eqref{approx2'} holds true under these circumstances.\\

The proof of Theorem \ref{thm:kis1} was based on the fact, that for every $y\in B_{\mathbb R}$, we can easily find an element
$x_y\in B_{\mathbb R^d}$, such that $\hat a\cdot x_y=y.$ It is enough to consider $x_y=\hat a^T y.$
In the case of a general convex set $\Omega$, we first need to define for any $\hat a \in \mathbb S^{d-1}$ fixed, a
function $x_\cdot:\hat a(\Omega+\bar \epsilon) \to \Omega+\bar \epsilon$ given by $y \mapsto x_y$, and such that
$$
\hat a \cdot x_y = y.
$$
In particular, for all $y \in \hat a(\Omega+\bar \epsilon)$ we need to find
$$
x_ y \in \Omega+\bar \epsilon \cap \{x \in \mathbb R^d: \hat a \cdot x = y \}.
$$
Since both $\Omega+\bar \epsilon$ and the solution space $\{x \in \mathbb R^d: \hat a \cdot x = y \}$ 
are closed convex sets in $\mathbb R^d$, one could use an alternating projection algorithm for finding 
$x_y$ \cite{babo96}. Thus, we can assume that, at least algorithmically, this map can be computed.
Moreover, and alternatively, since the operation described above, i.e., finding $x_y\in B_{\mathbb R^d}$, such that $\hat a\cdot x_y=y$,
has to be executed as many times as we need to define, e.g., an appropriate spline approximation of $\hat g$, we may proceed as follows:
we find first $x_{\max}, x_{\min} \in B_{\mathbb R^d}$, such that $\hat a\cdot x_{\max}=\max_{x \in B_{\mathbb R^d}} \hat a\cdot x$ and
$\hat a\cdot x_{\min}=\min_{x \in B_{\mathbb R^d}} \hat a\cdot x$. 
Then any other $x_y$ such that $y =\hat a \cdot x_y$ is computed very fast by $x_y = \lambda_y x_{\min}+ (1- \lambda_y) x_{\max}$
for some $\lambda_y \in [0,1]$.

With this modification, also Theorem \ref{thm:kis1} holds true, with the definition of $\hat g$ given in Algorithm 1 replaced now by
\begin{equation*}
\hat g(y) := f(x_y), \quad y \in  \hat a(\Omega+\bar \epsilon)
\end{equation*}
and \eqref{error:kis1} replaced by
\begin{equation*}
\|f-\hat f \|_\infty \leq 2 C_2 (\diam(\Omega)+2\bar\epsilon) \frac{\nu_1}{ \sqrt{\alpha(1-s) }-\nu_1}.
\end{equation*}
Unfortunately, and this seems to be the main drawback of this approach, the diameter of $\Omega$, $\diam(\Omega)= \max_{x,x' \in \Omega} \|x-x'\|_{\ell_2^d}$ may grow with $d$.
This is especially the case, when $\Omega=[-1,1]^d$, which gives $\diam(\Omega)=\sqrt{2d}.$ 

\subsection{An approach through Minkowski functional}
To get better results for specific convex bodies (i.e. $\Omega=[-1,1]^d$), we propose another approach. We stress very clearly that up to now this is only to be understood as an open direction, which is a subject of further research.\\
We assume, that $\Omega$ is a closed convex set, which is \emph{absorbing} and \emph{balanced}, i.e.
\begin{itemize}
\item for every $x\in\mathbb R^d$, there is a $t=t(x)>0$, such that $tx\in\Omega$,
\item $\alpha\Omega:=\{\alpha x: x\in\Omega\}\subset \Omega$ for every $\alpha\in [-1,1]$.
\end{itemize}
Then we can define its Minkowski functional as
$$
p_{\Omega}(x):=\inf\{r>0:x/r\in\Omega\},\quad x\in\mathbb R^d.
$$
It is well known, that this expression is actually a norm and $\Omega$ is its unit ball. Hence
\begin{equation}\label{eq:diam}
\sup_{x,x'\in\Omega}p_{\Omega}(x-x')\le 2.
\end{equation}
This allows us to replace the inequality
$$
|(a-\hat a)\cdot (x_y-x)|\le \|a-\hat a\|_2\cdot \|x_y-x\|_2
$$
by
$$
|(a-\hat a)\cdot (x_y-x)|\le \|a-\hat a\|'_{\Omega}\cdot \|x_y-x\|_\Omega.
$$
Here, $\|\cdot\|_\Omega=p_\Omega(\cdot)$ and $\|\cdot\|'_{\Omega}$ is its dual norm. According to \eqref{eq:diam},
this solves the problem of the factor $\diam (\Omega)$ - the diameter of $\Omega$ with respect to $\|\cdot\|_{\Omega}$ is always bounded by 2. Unfortunately the problem is transferred to the second factor, namely $\|a-\hat a\|'_{\Omega}$.
For this, one would need the analogue of Theorem \ref{thm:cs} with the $\ell_2^d$-norm in \eqref{woj1} replaced 
by $\|\cdot\|'_{\Omega}.$ While any treatment of this general case is clearly beyond the scope of this paper and 
remains a subject of further investigation, we can shortly sketch what happens in the special case $\Omega=[-1,1]^d$. 
Then we simply have $\|\cdot\|_\Omega=\|\cdot\|_{{\ell_\infty^d}}$ and $\|\cdot\|'_\Omega=\|\cdot\|_{{\ell_1^d}}$. 
To estimate $\|a-\hat a\|_{\ell_1^d}$ we would have to combine Lemma~3.1 in \cite{codadekepi10} with \eqref{approx2'} 
and would get again a result that does not depend on the dimension $d$.

\subsubsection*{Acknowledgments}
Massimo Fornasier would like to thank Ronald A. DeVore for his kind and warm hospitality at Texas A$\&$M University and the 
very exciting daily joint discussions which later inspired part of this work.
We acknowledge the financial support provided by the START-award ``Sparse Approximation and Optimization in High Dimensions'' 
of the Fonds zur F\"orderung der wissenschaftlichen Forschung (FWF, Austrian Science Foundation).
We would like to thank the anonymous referees for their very valuable comments and remarks.

\bibliography{FS}

\providecommand{\bysame}{\leavevmode\hbox to3em{\hrulefill}\thinspace}
\providecommand{\MR}{\relax\ifhmode\unskip\space\fi MR }
\providecommand{\MRhref}[2]{%
  \href{http://www.ams.org/mathscinet-getitem?mr=#1}{#2}
}
\providecommand{\href}[2]{#2}
\begin{thebibliography}{10}

\bibitem{ahwi02}
R.~Ahlswede and A.~Winter, \emph{Strong converse for indentification via
  quantum channels}, {I}{E}{E}{E} {T}rans. {I}nform. {T}heory \textbf{48}
  (2002), no.~3, 569--579.

\bibitem{cael11}
{E}. {A}rias{--}{C}astro and {Y}.~{C}. Eldar, \emph{Noise folding in compressed
  sensing}, IEEE Signal Processing Letters \textbf{18} (2011), no.~8, 478--481.

\bibitem{badadewa08}
{R}.~{G}. {B}araniuk, {M}. {D}avenport, {R}.~{A}. {D}e{V}ore, and {M}. {W}akin,
  \emph{{A} simple proof of the restricted isometry property for random
  matrices}, {C}onstr. {A}pprox. \textbf{28} (2008), no.~3, 253--263.

\bibitem{babo96}
{H.} {B}auschke and {H.} {B}orwein, \emph{{On projection algorithms for solving
  convex feasibility problems}}, SIAM Review \textbf{38} (1996), no.~3,
  367--426.

\bibitem{ca99}
E.~J. Cand\`es, \emph{{Harmonic analysis of neural networks.}}, Appl. Comput.
  Harmon. Anal. \textbf{6} (1999), no.~2, 197--218.

\bibitem{ca03}
\bysame, \emph{{Ridgelets: Estimating with ridge functions.}}, Ann. Stat.
  \textbf{31} (2003), no.~5, 1561--1599.

\bibitem{cado99}
{E}.~{J}. {C}and{\`e}s and {D}.~{L}. {D}onoho, \emph{{R}idgelets: a key to
  higher-dimensional intermittency?}, {P}hilos. {T}rans. {R}. {S}oc. {L}ond.
  {S}er. {A} {M}ath. {P}hys. {E}ng. {S}ci. \textbf{357} (1999), no.~1760,
  2495--2509.

\bibitem{carota06-1}
{E}.~{J}. {C}and{\`e}s, {J}. {R}omberg, and {T}. {T}ao, \emph{{S}table signal
  recovery from incomplete and inaccurate measurements}, {C}omm. {P}ure {A}ppl.
  {M}ath. \textbf{59} (2006), no.~8, 1207--1223.

\bibitem{cata07}
{E}.~{J}. {C}and\`es and {T}. {T}ao, \emph{The {D}antzig selector: statistical
  estimation when $p$ is much larger than $n$}, Ann. Statist. \textbf{35}
  (2007), no.~6, 2313--2351.

\bibitem{codade09}
{A}. {C}ohen, {W}. {D}ahmen, and {R}.~{A}. {D}e{V}ore, \emph{{C}ompressed
  sensing and best k-term approximation}, {J}. {A}mer. {M}ath. {S}oc.
  \textbf{22} (2009), no.~1, 211--231.

\bibitem{codadekepi10}
{A}. {C}ohen, {I}. Daubechies, {R}.~{A}. {D}e{V}ore, {G}. Kerkyacharian, and
  {D}. Picard, \emph{Capturing ridge functions in high dimensions from point
  queries}, Constr. Approx. (to appear).

\bibitem{cohi62}
{R.} {C}ourant and {D.} {H}ilbert, \emph{{Methods of Mathematical Physics,
  II}}, {New York: Interscience Publishers}, 1962.

\bibitem{depewo09}
{R}.~{A}. {D}e{V}ore, {G}. {P}etrova, and {P}. {W}ojtaszczyk, \emph{Instance
  optimality in probability with an $\ell_1$-minimization decoder}, {A}ppl.
  {C}omput. {H}armon. {A}nal. \textbf{27} (2009), no.~3, 275--288.

\bibitem{do06-2}
{D}.~{L}. {D}onoho, \emph{{C}ompressed sensing}, {I}{E}{E}{E} {T}rans.
  {I}nform. {T}heory \textbf{52} (2006), no.~4, 1289--1306.

\bibitem{fa02-2}
{M}. {F}azel, \emph{{Matrix Rank Minimization with Applications}}, Ph.D.
  thesis, {S}tanford {U}niversity, 2002.

\bibitem{fo10}
{M}. {F}ornasier, \emph{{N}umerical {M}ethods for {S}parse {R}ecovery},
  {T}heoretical {F}oundations and {N}umerical {M}ethods for {S}parse {R}ecovery
  ({M}. {F}ornasier, ed.), {R}adon {S}eries on {C}omputational and {A}pplied
  {M}athematics, {D}e {G}ruyter {V}erlag, 2010.

\bibitem{fora10}
{M}. {F}ornasier and {H}. {R}auhut, \emph{Compressive {S}ensing}, Handbook of
  Mathematical Methods in Imaging ({O}. {S}cherzer, ed.), vol.~1, Springer,
  2010, pp.~187--229.

\bibitem{fo09}
{S.} {F}oucart, \emph{A note on ensuring sparse recovery via
  {$\ell_1$}-minimization}, Appl. Comput. Harmon. Anal. \textbf{29} (2010),
  no.~1, 97--103.

\bibitem{gova96}
{G}. {G}olub and {C}.~{F}. van {L}oan, \emph{{M}atrix {C}omputations}, 3rd ed.,
  {T}he {J}ohns {H}opkins {U}niversity {P}ress, 1996.

\bibitem{jo55}
{F.} {J}ohn, \emph{{Plane waves and spherical means applied to partial
  differential equations}}, {New York: Interscience Publishers}, 1955.

\bibitem{le01}
{M}. {L}edoux, \emph{{The Concentration of Measure Phenomenon}}, {American
  Mathematical Society, Providence, Rhode Island}, 2001.

\bibitem{losh75}
{B.}~{F.} {L}ogan and {L.}~{A.} {S}hepp, \emph{Optimal reconstruction of a
  function from its projections}, Duke Math. J. \textbf{42} (1975), no.~4,
  645--659.

\bibitem{nowo08}
{E}. {N}ovak and {H}. {W}o{\'z}niakowski, \emph{{T}ractability of
  {M}ultivariate {P}roblems, {V}olume {I}: {L}inear {I}nformation}, {EMS Tracts
  in Mathematics, Vol. 6}, {Eur. Math. Soc., {Z\"u}rich}, 2008.

\bibitem{nowo09}
\bysame, \emph{{A}pproximation of infinitely differentiable multivariate
  functions is intractable}, Journal of Complexity \textbf{25} (2009),
  398--404.

\bibitem{ol10}
R.~I. Oliveira, \emph{Sums of random {H}ermitian matrices and an inequality by
  {R}udelson}, arXiv:1004.3821, 2010.

\bibitem{oymofaha11}
{S}. {O}ymak, {K}. {M}ohan, {M}. {F}azel, and {B}. {H}assibi, \emph{{A}
  simplified approach to recovery conditions for low-rank matrices}, Preprint
  (2011).

\bibitem{pi99}
{A.} {P}inkus, \emph{{Approximation theory of the MLP model in neural
  networks}}, Acta Numerica \textbf{8} (1999), 143--195.

\bibitem{fapareXX}
{B}. {R}echt, {M}. {F}azel, and {P}. {P}arillo, \emph{{G}uaranteed minimum rank
  solutions to linear matrix equations via nuclear norm minimization},
  {S}{I}{A}{M} {R}ev. \textbf{52} (2010), no.~3, 471--501.

\bibitem{ruve07}
M.~Rudelson and R.~Vershynin, \emph{Sampling from large matrices: An approach
  through geometric functional analysis}, J. ACM \textbf{54} (2007), no.~4,
  Art. 21, 19 pp.

\bibitem{ru80}
{W}. {R}udin, \emph{{Function theory in the unit ball of ${\mathbb {C}}^{n}$}},
  {Springer-Verlag, New York-Berlin}, 1980.

\bibitem{scvy11}
{K}. {S}chnass and {J}. {V}yb{\'i}ral, \emph{{C}ompressed learning of
  high-dimensional sparse functions}, ICASSP11, 2011.

\bibitem{st90}
{G}.~{W}. {S}tewart, \emph{{P}erturbation theory for the singular value
  decomposition}, {SVD} and {S}ignal {P}rocessing, II ({R}.~{J}. {V}acarro,
  ed.), {E}lsevier, 1991.

\bibitem{tr10}
{J}. {T}ropp, \emph{{U}ser-friendly tail bounds for sums of random matrices},
  arXiv:math.PR 1004.4389v6 (2010).

\bibitem{we72}
{P}.-{A}. {W}edin, \emph{{Perturbation bounds in connection with singular value
  decomposition}}, {BIT} \textbf{12} (1972), 99--111.

\bibitem{we12}
{H}. {W}eyl, \emph{{Das asymptotische Verteilungsgesetz der Eigenwerte linearer
  partieller Differentialgleichungen (mit einer Anwendung auf die Theorie der
  Hohlraumstrahlung)}}, {M}athematische {A}nnalen \textbf{71} (1912), 441--479.

\bibitem{woxx}
{P}. {W}ojtaszczyk, \emph{$\ell_1$ minimisation with noisy data}, preprint
  (2011).

\end{thebibliography}

\end{document}